\documentclass[a4paper,11pt]{amsart}
\usepackage[hyperindex,breaklinks]{hyperref}
\usepackage[english]{babel}
\usepackage{amsmath,amssymb,amsthm}
\usepackage{amsfonts}
\usepackage{hyperref}
\usepackage[centering]{geometry}
\usepackage{verbatim,amscd,amsmath,amsthm,amstext,latexsym,lscape,rawfonts}
\usepackage[arrow,matrix,curve]{xy}
\usepackage{pb-diagram}

\usepackage{xcolor}

\newtheorem{theorem}{Theorem}[section]
\newtheorem{corollary}[theorem]{Corollary}
\newtheorem{lemma}[theorem]{Lemma}
\newtheorem{proposition}[theorem]{Proposition}
\theoremstyle{definition}
\newtheorem{definition}{Definition}

\newtheorem{remark}{Remark}
\newtheorem*{acknow}{Acknowledgments}

\numberwithin{equation}{section}

\newcommand{\R}{\mathbb{R}}
\newcommand{\C}{\mathbb{C}}
\newcommand{\D}{\mathbb{D}}
\newcommand{\St}{\mathbb{S}}
\newcommand{\F}{\mathbb{F}}
\newcommand{\E}{\mathbb{E}}
\newcommand{\Z}{\mathbb{Z}}
\newcommand{\f}{\mathfrak{f}}

\newcommand{\di}{\operatorname{diag}}
\newcommand{\tr}{\operatorname{tr}}

\begin{document}

\title[Minimal Lagrangian surfaces in $\mathbb{C}P^2$ : Part II ]
{Minimal Lagrangian surfaces in $\mathbb{C}P^2$ via the loop group method   Part II: The general case}

\author{Josef F.~Dorfmeister}
\address{Fakult\"{a}t F\"{u}r Mathematik, TU-M\"{u}nchen, Boltzmann Str. 3,
D-85747, Garching, Germany}
\email{dorfm@ma.tum.de}
\author{Hui Ma}
\address{Department of Mathematical Sciences, Tsinghua University,
Beijing 100084, P.R. China} \email{ma-h@mail.tsinghua.edu.cn}

\begin{abstract} 
 We extend the techniques introduced in \cite{DoMaB1} for contractible Riemann surfaces to construct minimal Lagrangian immersions from arbitrary Riemann surfaces into $\mathbb{C}P^2$ via  the loop group method.
Based on the potentials of translationally equivariant minimal Lagrangian surfaces, we introduce perturbed equivariant  minimal Lagrangian surfaces in $\C P^2$ and 
construct  a  class of minimal Lagrangian cylinders. Furthermore, we show that these minimal Lagrangian cylinders  approximate  Delaunay cylinders with respect to some weighted Wiener norm of the twisted loop group $\Lambda SU(3)_{\sigma}$. 
\end{abstract}

\maketitle
 
%%%%%%%%%%%%%%%%%%%%%
\tableofcontents

\section{Introduction}

This paper is a continuation of Part I (\cite{DoMaB1}), where we explored minimal Lagrangian surfaces defined on a contractible Riemann surface using the loop group method \cite{DPW}.
Notably, any such a minimal Lagrangian surface admits a horizontal lift to $S^5$. 
Consequently,  the investigation of  the minimal Lagrangian surfaces presented in \cite{DoMaB1}  is equivalent to  studying  minimal Legendrian surfaces in $S^5$.

However, minimal Lagrangian immersions in $\C P^2$ do not admit  a global horizontal (i.e. Legendrian) lift to $S^5$ in general. Instead, a minimal Lagrangian surface $f: M \rightarrow \mathbb{C}P^2$ either has a horizontal lift or can be obtained as a quotient of a minimal Lagrangian surface possessing an order three  symmetry (see Theorem \ref{covering to horizontal mLi}).  In this paper we extend the techniques introduced in \cite{DoMaB1}  to minimal Lagrangian immersions of arbitrary Riemann surfaces into $\mathbb{C}P^2$.

Among all minimal Lagrangian surfaces, the most beautiful class is the one with symmetries.
For a minimal Lagrangian immersion $f: M\rightarrow \mathbb{C} P^2$,  a symmetry means a pair $(\gamma, \mathcal{R}) \in
(\mathrm{Aut}(M), \mathrm{Iso}_0(\mathbb{C} P^2))$ such that 
$f(\gamma \cdot z) = \mathcal{R}  f(z)$ holds for all $z \in M$.  For each full minimal Lagrangian immersion $f$ with symmetry, its natural lift  $\tilde{f}$ to the universal cover of $M$ admits a symmetry $(\tilde{\gamma}, \mathcal{R})\in (\mathrm{Aut}(\mathbb{D}), \mathrm{Iso}_0(\mathbb{C} P^2))$. When $\tilde{\gamma}$ normalizes $\pi_1(M)$,  it descends a symmetry on $M$.  Moreover, under the symmetry of $f$, the horizontal lift, the extended frame, and the potential transform accordingly. Thus, starting from an arbitrary Riemann surface  different from $S^2$ and a  potential satisfying appropriate transformation formulas defined from $\pi_1(M)$, one can construct a minimal Lagrangian immersion into $\C P^2$. 

However, the transformation formulas are not easy to handle. Fortunately, one can prove the existence of invariant holomorphic (respectively, meromorphic) potentials in both cases where the Riemann surface is non-compact or compact.
This implies that one can actually carry out the loop group construction of all minimal Lagrangian immersions by starting from some potential where the gauged matrix is $I$. 
 Thus only the closing conditions for the generators of the fundamental group need to be discussed, which is, of course, in general a very difficult task.

In particular, an immersion admitting  a 1-parameter group $(\gamma_t, R_t)$ of symmetries is called an equivariant immersion. By classical complex analysis, we know that there are two types of equivariant surfaces:  translationally equivariant surfaces and rotationally equivariant surfaces. It turns out that any rotationally equivariant minimal Lagrangian surfaces are either standard examples or  can be obtained from some periodic translationally equivariant minimal Lagrangian surfaces. Furthermore, the rotationally equivariant immersions of type $(R0)$ (refer to Section \ref{subsec:6.1}) correspond in a one-to-one relation to the translationally equivariant minimal Lagrangian $2 \pi$-periodic immersions.
And translationally equivariant minimal Lagrangian surfaces in $\C P^2$ are the simplest non-trivial examples. They have been investigated very well and can be constructed by the loop group method (\cite{DoMaNL, DoMaEX}).

It is noteworthy that one of the goals of the loop group method is to provide a tool for the explicit construction of surfaces, here minimal Lagrangian surfaces.  This approach helps  to find potentials  that lead to  minimal Lagrangian surfaces endowed with certain  desired additional properties. 
We first consider minimal Lagrangian immersions from Riemann surfaces with abelian fundamental group 
(refer to \cite{Farkas-Kra} or Section \ref{sec:abelian pi_1}).
Besides equivariant minimal Lagrangian surfaces (\cite{DoMaNL, DoMaEX}),  examples 
for the following cases are presented:
\begin{enumerate}
\item all minimal Lagrangian surfaces of contractible Riemann surfaces with abelian symmetry group.
By the results in  Section \ref{subsec:construction} ,  these can be constructed from periodic or doubly periodic potentials.
\item all minimal Lagrangian surfaces defined on  non-compact Riemann surfaces with abelian fundamental group $\Z$.
By the results in Sections \ref{subsec:non-compact-C} and \ref{subsec:non-compact-H} (See Theorem \ref{Thm 5.10}),  these can be constructed from periodic potentials, defined on a contractible Riemann surface, for which the monodromy representation is trivial for some $\lambda \in S^1.$

\item  all minimal Lagrangian tori. By the results  of Section \ref{sec:compact} these can be constructed from doubly periodic potentials, defined on $\C$, for which the monodromy representation is trivial for some $\lambda \in S^1$.
Alternatively, there exists a different  construction method, called finite type method, by using algebraic geometry described in  \cite{MM} and \cite{McIn}.
\end{enumerate}

 Inspired by the work \cite{DW08} on the CMC case, in this paper we perturb the potentials of rotationally equivariant immersions of type $(R0)$ and 
define perturbed Delaunay minimal Lagrangian surfaces. For the sake of simplicity, we further impose more restrictions to introduce $*$-perturbed Delaunay potentials and corresponding surfaces. For a given $*$-perturbed Delaunay potential $\eta$ with a single singular point, by adapting the method of \cite{Sch-Sch}, we solve the complex differential equations $dH=H\eta$, which yields a  minimal Lagrangian cylinder eventually. With respect to the topology induced from certain weighted Wiener norm, we show that $*$-perturbed minimal Lagrangian cylinder in $\C P^2$  approximate minimal Lagrangian Delaunay cylinders. 

 The paper is organized as follows: Section \ref{sec:2} provides a review of basic notions and results on minimal Lagrangian surfaces in $\mathbb{C} P^2$ and the loop group method.
 Section \ref{sec:3} examines  the liftablity of a minimal Lagrangian immersion.
 In Section \ref{sec:4}  symmetries for minimal Lagrangian immersions from arbitrary Riemann surfaces are discussed, including and their effects on horizontal lifts and extended frames.
 Additionally, the construction of minimal Lagrangian immersions from a general Riemann surface $M$ different from $S^2$ is outlined. 
 Section \ref{sec:inv frame} demonstrates  the existence of an  invariant  holomorphic potential (resp. meromorphic potential) for any minimal Lagrangian map to $\mathbb{C} P^2$ from any non-compact (resp. compact) Riemann surface which is invariant under $\pi_1(M)$.  
 This shows that any  minimal Lagrangian
immersion  from some Riemann surface $M$ into $\mathbb{C} P^2$
can be obtained by the loop group method from some 
invariant holomorphic potential or meromorphic potential defined on the universal cover of $M$. 
 The final section covers the characterization of equivariant minimal Lagrangian surfaces in $\mathbb{C}P^2$, the construction of $*$-perturbed equivariant minimal Lagrangian surfaces in $\mathbb{C} P^2$, and the description of their asymptotic behavior. The uniqueness of $*$-perturbed equivariant minimal Lagrangian surfaces with respect to the potentials is also discussed in this section.The Appendix explains the norms used for the loop groups, including their properties, which are utilized in the asymptotic estimates of Section  \ref{sec:pertubed}.

%%%%%%%%%%%%%%%%%%%%%
\section{Recalling notation and some basic results}\label{sec:2}

In this section we recall notation and results from \cite{DoMaB1}.
All statements there were made under the assumption that  the
domain of a minimal Lagrangian immersion is a contractible open subset of $\C$. 
For details we refer to \cite{DoMaB1}.

\subsection{Minimal Lagrangian surfaces in \texorpdfstring{$\mathbb{C}P^2$}{} }
\label{subsec:miL}
%%%%%%%%%%%%%%%%%%%%%%%%%%%%%

Suppose that $f:M \rightarrow \C P^2$ is a minimal Lagrangian immersion from a 
Riemann surface
in the complex projective plane endowed with the Fubini-Study metric of constant holomorphic sectional curvature $4$.
The induced metric on $M$ generates a conformal structure with
respect to which the metric is $g=2e^{u} dzd{\bar z}$, and where $z=x+iy$ is a local
conformal coordinate on $M$ and $u$ is a real-valued function defined on $M$ locally.

It is useful to point out that any  minimal Lagrangian immersion $f:M\rightarrow {\mathbb C}P^2$ 
has a natural lift  $\tilde{f}: \tilde{M} \rightarrow {\mathbb C}P^2$ as a  minimal Lagrangian immersion over the universal covering $\tilde{M}$ of $M$.
 
For our approach in this paper we will need certain lifts to  $S^5 =\{Z\in \mathbb{C}^3 | Z\cdot \bar{Z}=1\}$, 
where $Z\cdot \overline{W}=\sum_{k=1}^{3} z_{k} \overline{w_{k}}$
denotes the Hermitian inner product for any $Z=(z_1,z_2,z_3)$ and $W=(w_1,w_2,w_3)\in \mathbb{C}^3$. Write $Z\cdot W=\langle Z, W\rangle + \sqrt{-1}\langle Z, JW\rangle$, where $\langle \,,\, \rangle$ is the Riemannian metric of $\C^3$ and $J$ is the standard complex structure of $\C^3$.

\begin{definition}
a) Let $M$ be an arbitrary Riemann surface and  $\f: M \rightarrow S^5$ an immersion.
Then $\f$ is called  
\emph{Legendrian}  if it satisfies 
$ d\f\cdot  \bar{\f}=0$.

b) Let $M$ be an arbitrary Riemann surface and  $f: M \rightarrow  {\mathbb C}P^2$  an immersion.
Then a map $\f: M \rightarrow S^5$ is called 
a horizontal lift of $f$ if and only if 
$f = \pi \circ \f$ and 
$ d\f\cdot  \bar{\f}=0$.
\end{definition}

\begin{remark} 
Recall that a Legendrian submanifold $L$ in a contact manifold $(M^{2n+1}, \alpha)$ is an $n$-dimensional submanifold of $M$ on which the contact form $\alpha$ vanishes.  The restriction of the one form $\sigma(Z)=\langle dZ, JZ\rangle$ on $\mathbb{C}^3$  to $S^5$ gives the standard contact form $\alpha$ on $S^5$. Notice
$dZ\cdot \bar{Z}=\langle dZ, Z\rangle+\sqrt{-1} \langle dZ, JZ\rangle.$ 
Hence $\alpha(Z)=-\sqrt{-1} dZ\cdot \bar{Z}$ for $Z\in S^5$.
Thus $\f: M \rightarrow S^5$ is Legendrian if $d\f\cdot \bar{\f}=0$.
\end{remark}

The following result leads to the local existence and uniqueness of a horizontal lift of a minimal Lagrangian surface in $\C P^2$. 

 \begin{theorem} [\cite{DoMaB1}, Theorem 2.5]
Let $M $ be a contractible Riemann surface and  $f: M \rightarrow {\mathbb C}P^2$ a Lagrangian immersion.
Then  $f$  has a horizontal lift ${\f}: M \rightarrow S^5 $, i.e. ${\f}$ satisfies the  equations
\begin{equation}\label{horizontal}
\f_z \cdot \overline{\f} = \f_{\bar z} \cdot \overline {\f} =0.
\end{equation}
In particular, ${\f}: M \rightarrow S^5 $ is Legendrian.
Moreover, $\f$ is uniquely determined by this property up to a constant factor $\delta \in S^1$. 
\end{theorem}

\begin{corollary} [\cite{DoMaB1}, Proposition 2.6]\label{Cor:globle-lift}
Let $f:M\rightarrow {\mathbb C}P^2$ be a minimal Lagrangian immersion of an oriented
surface and $\tilde{f}:\tilde{M} \rightarrow {\mathbb C}P^2$  its natural lift to the universal cover $\tilde{M}$ of $M$. 
If  $M \neq S^2$, then $\tilde{M} $ is contractible and 
$\tilde{f}:\tilde{M} \rightarrow {\mathbb C}P^2$  can be lifted to a horizontal map
$\tilde{\f}:\tilde{M} \rightarrow S^5$ and this map is uniquely determined up to some constant factor $\delta \in S^1$.
\end{corollary}

%%%%%%%%%%%%%%%%%%%%
 For general Riemann surfaces $M$  and general conformal immersions 
 $f: M \rightarrow \C P^2$ such a (global) lift  $\f:M \rightarrow S^5$ may 
 not exist.
 Actually, the well known Clifford torus does not admit a global horizontal lift \cite{DoMaB1}, example in section 2.4.
 
Since any minimal Lagrangian sphere in $\mathbb{C}P^2$  is totally geodesic, 
and thus congruent to a piece of $\R P^2 \subset \C P^2,$ as in 
 \cite{DoMaB1}  
we exclude $M =S^2$ throughout this  paper.  
Thus in this paper we will first lift any minimal Lagrangian surface  $f: M \rightarrow\C P^2$
naturally to a minimal surface $\tilde{f}: \D \rightarrow \C P^2$, where $\D$ denotes the (contractible) 
universal cover of $M$, and then consider the global  horizontal lift $\tilde{\f} : \D \rightarrow S^5$ of $\tilde{f}$.

%%%%%%%%%%%%%%%%%%%%%%%%%%%%%%
\subsection{Frames for horizontal lifts}\label{subsec:frames}
%%%%%%%%%%%%%%%%%%%%%%%%%%%%%%
In this subsection we assume that $\D$ is a contractible Riemann surface   (i.e. the unit disk $\mathcal{D}$ or the complex plane $\mathbb{C}$), $f: \D \rightarrow \C P^2$ a minimal Lagrangian immersion and $\f : \D  \rightarrow S^5$ a horizontal lift of $f$.

The fact that the induced metric $g$ is conformal is equivalent to
\begin{equation}\label{fconf1} 
\begin{split}
&{\f}_z\cdot \overline{\f_z}=\f_{\bar z}\cdot \overline{\f_{\bar z}}=e^{u},\\
& \f_z\cdot \overline{\f_{\bar z}}=0.
\end{split}
\end{equation}
Thus 
\begin{equation} \label{coordframe}
\mathcal{F}=(e^{-\frac{u}{2}}\f_z, e^{-\frac{u}{2}}\f_{\bar z}, \f),
\end{equation}
called the {\it coordinate frame }of $f$,  defines a Hermitian orthonormal moving frame on the surface $\D$.

It follows from \eqref{horizontal}, \eqref{fconf1} and the minimality of
 $f$ that $\mathcal{F}$ satisfies the frame equations (see e.g. \cite{MM})
\begin{equation}\label{eq:frame1}
\mathcal{F}_{z}=\mathcal{F} {\mathcal U}, \quad \mathcal{F}_{\bar z}=\mathcal{F} {\mathcal V},
\end{equation}
where
\begin{equation}\label{eq:UV1}
{\mathcal U}=\left(\begin{array}{ccc}
                   \frac{u_z}{2} & 0 & e^{\frac{u}{2}} \\
      e^{-u}\psi  &-\frac{u_z}{2}  & 0 \\
       0 & -e^{\frac{u}{2}}  &0 \\
     \end{array}
   \right),
   \quad
{\mathcal V}=\left(
     \begin{array}{ccc}
      -\frac{u_{\bar z}}{2}  &  - e^{-u}\bar\psi  &0 \\
       0& \frac{u_{\bar z}}{2} &e^{\frac{u}{2}} \\
       -e^{\frac{u}{2}}  &0 &0 \\
     \end{array}
   \right),
\end{equation}
with
\begin{equation}\label{eq:phipsi}
\psi=\f_{zz}\cdot\overline{\f_{\bar z}}.
\end{equation}

Using (\ref{fconf1}) one can easily check that the cubic differential $\Psi=\psi dz^3$ is actually independent of the choice of the lift $\f$ and the complex coordinate $z$ of $\D$ and thus is globally defined on the Riemann surface $\D$. The differential $\Psi$ is called the \emph{Hopf differential} of $f$.

The compatibility condition of  the equations \eqref{eq:frame1} is
${\mathcal U}_{\bar z}-{\mathcal V}_z =[{\mathcal U},{\mathcal V}]$,  
and using \eqref{eq:UV1} this turns out to be equivalent to
 \begin{align}
u_{z\bar z}+e^u-e^{-2u}|\psi|^2&=0,\label{eq:mLsurfaces}\\
\psi_{\bar z}&=0.\label{eq:Codazzi}
\end{align}

%%%%%%%%%%%%%%%

Since  $\f$   is uniquely determined up to a constant factor $\delta \in S^1$,
also $\mathcal{F}$ is only defined up to this  constant factor. 

We can actually  have 

\begin{proposition} (\cite{DoMaB1}, Proposition 2.7) \label{prop:constDet}
Let $\D$ be a contractible Riemann surface  and  $f: \D \rightarrow {\mathbb C}P^2$ a minimal Lagrangian immersion. Then for the corresponding  frame $\mathcal{F}$ we can assume without loss of generality
$\det \mathcal{F} = -1$.
Under this assumption $\mathcal{F}$ is uniquely determined up to some factor $\delta$ satisfying $\delta^3 = 1$.
\end{proposition}

From here on we will always assume that $ \det \mathcal{F} = -1$ holds.
%%%%%%%%%%%%%%%%%%
\subsection{The loop parameter and extended frames}
\label{subsec:2.3}

From here on we will adopt loop groups, following a similar approach to those used for many surface classes.
The notation for minimal Lagrangian surfaces in $\C P^2$  has been introduced and discussed in \cite{DoMaB1}.
We provide a list of basic notation in the Appendix and briefly mention some relevant facts in the text below. However, for any detailed information and additional results, we refer the reader to \cite{DoMaB1}.

Introducing the \emph{spectral parameter}  $\lambda$ as usual (See for example, Subsection 2.5, \cite{DoMaB1}) one  obtains that 
one can also assume without loss of generality  $\det \mathcal{F}(z,\bar{z}, \nu) = -1$.  

It turns out to be convenient to consider in place of the  frames $\mathcal{F}(z,\bar{z}, \nu)$ the gauged frames 
\begin{equation} \label{F}
\mathbb{F}(\lambda)=\mathcal{F}(\nu)\left(
                                   \begin{array}{ccc}
                                     -i\lambda & 0 & 0 \\
                                     0 &  -i\lambda^{-1} & 0 \\
                                     0 & 0 &  1 \\
                                   \end{array}
                                 \right),
\end{equation}
where $i\lambda^3 \nu=1$. Note that we have  $\det \mathcal{F} = -1$ and  
$\det \F = 1$, 
hence  $\F \in \Lambda SU(3)$; for more details on this notation see the Appendix. 

 We will always assume $\mathbb{F}(z_0,\bar{z}_0, \lambda) = \mathrm{I}$ for some fixed base point $z_0$, unless the opposite is stated explicitly. With this normalization, the frame $\mathbb{F}$ is uniquely determined by $f$. Note that this normalization implies the normalizations 
 $\f (z_0, \bar{z}_0,\lambda) = e_3$ and $\mathbb{F}(z_0, \bar{z}_0,\lambda) = \mathrm{I}$ which determines these quantities uniquely.

For the frame  $\mathbb{F}$ we obtain the equations
\begin{equation}\label{eq:mathbbF}
\begin{split}
\mathbb{F}^{-1}\mathbb{F}_z&=
\frac{1}{\lambda}\left(
   \begin{array}{ccc}
     0 & 0 & i e^{\frac{u}{2}} \\
  -i \psi e^{-u}   & 0 & 0 \\
     0 &  i e^{\frac{u}{2}} & 0 \\
   \end{array}
 \right)+\left(
   \begin{array}{ccc}
   \frac{u_z}{2}  & &  \\
      & -\frac{u_z}{2} &  \\
      & & 0 \\
   \end{array}
 \right)\\
 &:=\lambda^{-1}U_{-1}+U_0,\\
\mathbb{F}^{-1}\mathbb{F}_{\bar{z}} &=\lambda \left(
   \begin{array}{ccc}
     0 &  -i\bar{\psi} e^{-u} & 0 \\
     0& 0   & i e^{\frac{u}{2}} \\
     i e^{\frac{u}{2}} &0 & 0 \\
   \end{array}
 \right)+\left(
   \begin{array}{ccc}
     -\frac{u_{\bar z}}{2} & &  \\
      & \frac{u_{\bar z}}{2}  & \\
     & & 0 \\
   \end{array}
 \right)\\
  &:=\lambda V_{1}+V_0.
\end{split}
\end{equation}

\begin{proposition}\label{prop:frame} \cite{MM, DoKoMa}
 Retaining the assumptions and the conventions above we consider a contractible Riemann surface $M $ and let $\mathbb{F}(z, \bar{z},\lambda),  \lambda\in S^1, z \in M, $ be a solution to the system \eqref{eq:mathbbF}.
Then $[\mathbb{F}(z, \bar{z}, \lambda)e_3]$ gives a minimal Lagrangian surface defined on $M$ with values in $\mathbb{C}P^2$ and 
with the metric $g=2e^{u}dzd\bar{z}$ and the Hopf differential $\Psi^{\nu}=\nu \psi dz^3$.

Conversely, suppose $f^{\nu}:M\rightarrow \mathbb{C}P^2$ is a conformal parametrization of
a minimal Lagrangian surface in $\mathbb{C}P^2$ with the metric $g=2e^{u}dzd\bar{z}$ and Hopf differential
$\Psi^{\nu}=\nu\psi dz^3$.  Then  on $M$ there exists  a  frame $\mathbb{F}: U \rightarrow SU(3)$
satisfying \eqref{eq:mathbbF}. This frame is unique if we choose a base point $z_0 \in M$ and normalize  $\F(z_0, \bar{z}_0,\lambda) = I$.
\end{proposition}

We have already pointed out that one can interpret the gauged frames $\F(z, \bar{z}, \lambda)$ above as elements of the loop group $\Lambda SU(3)$. It actually turns out that they belong to a smaller, twisted, loop group.

%%%%%%%%%%%%%%%%%%%%%%%%
\subsection{The loop group characterization for minimal Lagrangian surfaces}
%%%%%%%%%%%%%%%%%%%%%%%%

 Let $\sigma$ denote the  automorphism of $G^{\mathbb{C}}=SL(3,\mathbb{C})$ of order $6$ defined by
\begin{equation} \label{sigmaandtaugroup}
\sigma: g\mapsto P (g^t)^{-1} P^{-1}, \quad \text{ where } 
P=\left(
  \begin{array}{ccc}
    0 & \epsilon^2 & 0 \\
    \epsilon^4 & 0 & 0 \\
    0 &0 & 1\\
  \end{array}
\right), \text{ with } \epsilon=e^{\pi i/3}.
\end{equation}
Let $\tau$ denote the anti-holomorphic involution of $\mathfrak{g}^{\mathbb{C}}=SL(3,\mathbb{C})$ which defines  the real form $G=SU(3)$, 
given by $$\tau(g):=(\bar{g}^t)^{-1}.$$
Then on the Lie algebra level the corresponding automorphism $\sigma$ of order $6$ and 
the anti-holomorphic automorphism $\tau$
of $sl(3,\mathbb{C})$ are
\begin{equation} \label{sigma and tau Lie}
\sigma: \xi \mapsto -P\xi^t P^{-1}, \quad \tau: \xi \mapsto -\bar{\xi}^t.
\end{equation}

By $\mathfrak{g}_l$ we denote the $\epsilon^l$-eigenspace of $\sigma$ in  $\mathfrak{g}^{\mathbb C}$. 
Explicitly these eigenspaces are given as follows
\begin{equation*}
\begin{split}
\mathfrak{g}_0&=\left\{
                    \begin{pmatrix}
                    a &  &  \\
                     & -a &  \\
                     &  & 0 \\
                \end{pmatrix}
                \mid a\in \C
\right\},
\quad
\mathfrak{g}_1=\left\{\begin{pmatrix}
                    0 & b & 0 \\
                     0 & 0 & a \\
                     a & 0 & 0 \\
                  \end{pmatrix}\mid a,b\in\C
\right\},
\\
\mathfrak{g}_2&=\left\{
                  \begin{pmatrix}
                    0 & 0& a  \\
                     0& 0 & 0 \\
                     0& -a & 0 \\
                  \end{pmatrix} \mid a\in \C
              \right\},
\quad
\mathfrak{g}_3=\left\{
                  \begin{pmatrix}
                    a &  &  \\
                      & a &  \\
                      &  & -2a \\
                  \end{pmatrix} \mid a\in \C
              \right\},\\
\mathfrak{g}_4&=\left\{
                  \begin{pmatrix}
                    0 & 0 & 0  \\
                     0& 0 & a \\
                     -a & 0 & 0 \\
                  \end{pmatrix} \mid a\in \C
            \right\},
\quad
\mathfrak{g}_5=\left\{
                  \begin{pmatrix}
                    0 & 0 & a \\
                     b & 0 & 0 \\
                     0 & a & 0 \\
                  \end{pmatrix} \mid a, b\in \C 
                \right\}.
\end{split}
\end{equation*}
Remark that the automorphism $\sigma$ defines the $6$-symmetric space 
$SU(3)/U(1)$
and any minimal Lagrangian surface in $\mathbb{C}P^2$ frames a primitive harmonic map $\mathbb{F}: M\rightarrow SU(3)/U(1)$.

In view of (\ref{eq:mathbbF}) and the eigenspaces stated just above one is led to consider the twisting automorphism
\begin{equation}\label{sigma:twist}
(\hat{\sigma} (g))(\lambda) = \sigma ( g( \epsilon^{-1} \lambda))
\end{equation}
of $\Lambda SL(3,\C)$.
Then the $\sigma$-twisted loop groups are defined as fixed point sets of this twisting automorphism (See the Appendix). It is easy to verify

\begin{proposition}
The frames $\mathbb{F} (z, \bar{z}, \lambda)$ satisfying 
$\mathbb{F} (z_0, \bar{z}_0, \lambda) = I$ are elements of the twisted loop group
$\Lambda SU(3)_\sigma$. 
\end{proposition}

Using loop group terminology, we can state (refer to \cite{MM}):

\begin{proposition} \label{char-mLi}
Let $f: \D \rightarrow \mathbb{C}P^2$ be a conformal parametrization 
of a contractible Riemann surface and $\f$ its horizontal lift.
Then the following statements are equivalent:
\begin{enumerate}
\item $f$ is minimal Lagrangian.
\item There exists a frame $\mathbb{F}: \D \rightarrow SU(3)$ which is primitive harmonic relative to $\sigma$. 
\item  There exists an extend frame $\F: \D\rightarrow \Lambda SU(3)_{\sigma}$ such that $\mathbb{F}^{-1}d\mathbb{F} =(\lambda^{-1} U_{-1}+U_0)dz+(\lambda V_1+V_0)d\bar{z} 
\subset \Lambda su(3)_{\sigma}$ is a one-parameter family of flat connections for all $\lambda \in \C^*$.
\end{enumerate}
\end{proposition}

%%%%%%%%%%%%%%%%%%%%%%%%%%%%%%%
\subsection{The basic loop group method}
\label{subsection:loop method}
%%%%%%%%%%%%%%%%%%%%%%%%%%%%%%%

Let us start from some primitive harmonic map $f: \D \rightarrow G/K$ from a contractible Riemann surface to a $k$-symmetric space $G/K$ with respect to an order $k$ $(k\geq 2)$ automorphism of $G$ and let's consider an  extended frame  $\F:  \D \rightarrow \Lambda G_{\sigma}$ of $f$. 
Unless stated otherwise we will always assume $\F(0,0,\lambda) = e$.

While the frame $\F$ satisfies a non-linear integrability condition, the objects we construct next will trivially satisfy the integrability condition.
\vspace{2mm}

{\bf Construction 1: Holomorphic potentials}
\vspace{2mm}

For any extended frame $\F : \D\rightarrow \Lambda G_{\sigma}$ with $\F(0, 0)=e$, of some primitive harmonic map $f$ one can show that there exists a global matrix function
$h: \D\rightarrow \Lambda^{+}G_{\sigma}^{\mathbb C}$ solving the $\bar{\partial}$-problem
\begin{equation}
 h^{-1}\bar\partial h =-(\alpha_{\mathfrak k}^{\prime\prime}+\lambda \alpha_{\mathfrak m}^{\prime \prime} ), \quad\quad
 h(0)=e
\end{equation}
 over   $\D$, so that $C=\F h$ gives a \emph{holomorphic 
extended frame}. The Maurer-Cartan form of $C$,  
$\mu=C^{-1}\partial C$  is a $(1,0)$-form defined on $\D$ and takes values in 
\begin{eqnarray*}
\Lambda_{-1,\infty}&:=&\{\xi\in \Lambda \mathfrak{g}_{\sigma}^{\mathbb C} | \xi \text{ extends holomorphically to } \\
&& \quad\quad 0<|\lambda|<1 
\text{ with a simple pole at } 0\}\\
&=&\{\xi=\sum_{l\geq -1} \lambda^l \xi_l\in \Lambda \mathfrak{g}_{\sigma}^{\mathbb C}\}. 
\end{eqnarray*}
This differential 1-form on $\D$ is called a \emph{holomorphic potential} for $f$.

Conversely, starting from a holomorphic $(1,0)$-form $\eta=\sum_{l\geq -1} \lambda^l \eta_l$, one can first solve the ODE $dC=C\eta$, $C(0, \lambda) = e$ over  $\D$, where $C\in \Lambda G^{\mathbb C}_{\sigma}$. Performing an Iwasawa decomposition of $C$: $C={\mathbb F}V_{+}$,
where $\mathbb{F}=\mathbb{F}(z, \bar{z},  \lambda) \in \Lambda G_{\sigma}$
and $V_{+}=V_0+\lambda V_1+\lambda^2 V_2+\cdots \in \Lambda^+ G^{\mathbb C}_{\sigma}$, it turns out that $\F$ is the extended frame of some primitive harmonic map $f=\pi\circ \F_{\lambda=1}:\D\rightarrow G/K.$

\begin{remark}
Below we will frequently assume that the $k$-symmetric space $G/K$ is compact, since this is the case of main interest to this paper and since for non-compact $k$-symmetric spaces the Iwasawa decomposition is not global which would require additional remarks/cases at many places.
\end{remark}

Altogether we obtain

\begin{theorem}[\cite{DPW}]\label{Th:DPW-hol}
 Let $G/K$ be a compact $k$-symmetric space.
Let $f: \D\rightarrow G/K$ be a primitive harmonic map with $f(0,0)=eK$ and 
$\F: \D \rightarrow \Lambda G_{\sigma}$  an extended frame of $f$ satisfying $\F(0, 0,\lambda) = e$. 
Then there exists a matrix function 
$V_+: \D \rightarrow  \Lambda^{+}G_{\sigma}^{\mathbb{C}}$ such that 
$C = \F V_+$ is holomorphic in $z \in \D$  and $C(0,\lambda) = I$ and $V_+(0,\lambda) = I$ hold.
Then $\eta = C^{-1} dC \in \Lambda_{-1,\infty}$ is a holomorphic $(1,0)$-form on $\D$, called a \emph{holomorphic potential} for $f$.

 Conversely, given a  holomorphic $(1,0)$-form $\eta \in \Lambda_{-1,\infty}$ on $\D$ 
we obtain a map  $C:\D \rightarrow \Lambda G_{\sigma}^{\mathbb C}$, 
satisfying $dC = C \eta$ and $C(0,\lambda)=e$.
Performing an Iwasawa decomposition of $C$ we obtain an extended frame $\F: \D \rightarrow \Lambda G_{\sigma}$ of some  primitive harmonic map $f=\pi\circ \F_{\lambda=1}$ and  $\F(0, 0, \lambda) = I$ holds. 
\end{theorem}

 \begin{remark}  \label{general DPW}
 The procedure above is sometimes used in a generalized form, i.e. there are no conditions imposed with regard to initial conditions.  For instance, one can start from a potential $\eta$ defined on a (say contractible) domain $\D$ and integrate $dH = H \eta$ over $\D$, without prescribing initial conditions. 
 Subsequently, an Iwasawa decomposition $H = F V_+$ yields a frame  which induces a harmonic map 
 $f: \D \rightarrow G/K$. From this, one then derives a surface of the required type as before. However, the main goal of the loop group procedure is to construct new surfaces and to determine a class of potentials which produce such surfaces.
 In such a case it is usually helpful to fix initial conditions. In particular, fixing initial conditions allows  us to demonstrate a one-to-one correspondence between harmonic maps and certain potentials. 
 \end{remark}

{\bf Construction 2: Normalized potentials}
\vspace{2mm}

If one does not require that $C = \F V_+$ is necessarily holomorphic in $z$, but only meromorphic,  then by using the  Birkhoff decomposition  
$\F_- = \F V_+$, \cite{DPW} shows that any  harmonic map 
$f: \D\rightarrow G/K$ can  be obtained from a meromorphic potential of the form 
$ \F_-^{-1} d\F_- =\lambda^{-1}\mu_{-1}$, with $\mu_{-1}$ meromorphic on $\D$.
Note, in this step $\F_-$ is automatically meromorphic on $\D$.

The converse procedure follows the ``reverse pattern'' outlined above.
We collect the results for the meromorphic case by

\begin{theorem}[\cite{DPW}] Let $G/K$ be a compact $k$-symmetric space.
Let $f: \D\rightarrow G/K$ be a primitive harmonic map with $f(0, 0)=eK$ and $\F: \D \rightarrow \Lambda G_{\sigma}$  an extended frame of $f$ satisfying $\F(0, 0, \lambda) = e$. 
Then there exists a discrete subset $S\subset \D\backslash \{0\}$ such that for any point $z\in \D\backslash S$, the Birkhoff decomposition 
$\F(z,\cdot)=\F_{-}(z,\cdot) \F_{+}(z,\cdot)$ exists, with 
$\F_{-}(z,\cdot)\in \Lambda^{-}_{*}G_{\sigma}^{\mathbb{C}}$ and 
$\F_{+}(z,\cdot)\in \Lambda^{+}G_{\sigma}^{\mathbb{C}}$, and 
$\eta= \F_{-}(z,\lambda)^{-1}d\F_{-}(z,\lambda)$ is a $\mathfrak{m}^{\mathbb C}$-valued meromorphic $(1,0)$-form with poles in $S$ and which only contains the power $\lambda^{-1}$.

 Conversely, given a  $\mathfrak{m}^{\mathbb C}$-valued meromorphic $(1,0)$-form $\eta$ on $\D$ containing only the power $\lambda^{-1}$, for which the solution to $\F_{-}(z,\lambda)^{-1}d\F_{-}(z,\lambda)= \eta$ with $\F_{-}(0,\cdot)=e$ is meromorphic, we obtain a map  $\F_{-}:\D\backslash S \rightarrow \Lambda_{*}^{-} G_{\sigma}^{\mathbb C}$, where the discrete subset $S\subset \D\backslash\{0\}$ consists of the poles of $\eta$.
Performing an Iwasawa decomposition of $\F_-$ we obtain an extended frame $\F: \D\backslash S\rightarrow \Lambda G_{\sigma}$ of some  primitive harmonic map $f$  which satisfies $\F(0,\lambda) = e$. 

The two constructions explained in this theorem are inverse to each other. 
\end{theorem}

\begin{remark} \begin{enumerate}
\item The  $\mathfrak{m}^{\mathbb C}$-valued meromorphic $(1,0)$-form $\eta$  on $\D$,  unique after the choice of a base point, 
is called the {\it normalized potential of $f$ with the point $0$ as the reference point}. 

\item  We would like to emphasize  that $\bar\partial \F_{-}(z,\lambda)=0$ on $\D\backslash S$.  
 
\item  In the generality discussed above, a meromorphic potential $\eta$ generally does not yield a globally smooth minimal Lagrangian immersion.  To ensure global smoothness, one needs impose specific relations between the poles and zeros of the coefficients of the potential. For CMC surfaces in $\R^3$, see, for instance, \cite{DoHaMero}.

\end{enumerate}
\end{remark}

%%%%%%%%%%%%%%%%
\section{The liftability of minimal Lagrangian immersions}
\label{sec:3}

We have considered so far mainly full and horizontally liftable minimal Lagrangian immersions into $\C P^2$. In this section we will generalize our discussion to arbitrary full, possibly not liftable, minimal Lagrangian immersions into $\C P^2$.

%%%%%%%%%%%%%%%%%%%%%%%%
\subsection{The basic set-up} \label{subsection: basic set-up}
 \label{subsection: non-compact case}

In \cite{DoKoMa},  we discussed immersions $f: M \to \C P^2$  without complex points, and considered ``lifts" to an immersion $\f: M \to S^5$.
The question of liftability in this sense was addressed in the Appendix  of \cite{DoKoMa}.

However, in this paper, our focus is not on any lift but specifically on horizontal lifts. 
Therefore, in this section  as in the previous sections,  we discuss the ``horizontal liftability'' of an immersion 
$f: M \to \C P^2$ to a horizontal, i.e., Legendrian,  immersion $\f: M \to S^5$.
In the subsequent subsections, we adapt the argumentation presented in the Appendix of \cite{DoKoMa} to fit current context of our discussion. 
 
 %%%%%%%%%%%%%%%%%%%%%%%%%%%
\subsection{Non-compact minimal Lagrangian immersions into \texorpdfstring{$\C P^2$}{}  by means of \cite{DoKoMa}}
\label{sec:3.2 non-cpt}

 \begin{theorem}\label{Thm: normalized coordinate frame}
 Let $\D \subset \C$ be a simply-connected domain and $f:\D \rightarrow \C P^2$ a full minimal Lagrangian immersion.
 Let $\f_0: \D \rightarrow S^5$ be a horizontal lift of $f$ and $\mathcal{F}(\f_0 )$ 
 the corresponding
 coordinate frame.  
 Then 
 \begin{enumerate}
  \item[a)] There exists some constant $\delta \in  S^1$ such that 
 $\det  \mathcal{F}(\delta\f_0 ) = 1$.
 
 \item[b)] Any two horizontal lifts $\f_0$ and  $\f_1$ of $f$ for which   $\det  \mathcal{F}(\f_0 ) = 1$ and 
  $\det  \mathcal{F}(\f_1 ) = 1$ differ by  a cubic root of unity.
  
  \item[c)] If we also impose the condition
  $\mathcal{F}(\f_0 ) (z_0,\bar{z}_0) = \mathrm{I}$ for some fixed base point $z_0$, 
  then the horizontal lift is uniquely determined.
 \end{enumerate}
 \end{theorem}
 
 \begin{proof}
 a) Put $\delta_0 =  \det \mathcal{F}(\f_0 ).$ Then it follows, if $f$ is full, minimal and Lagrangian 
 that $\delta_0 : \D \rightarrow S^1$ is locally constant.
 Since $\D$ is simply-connected, $\delta_0$ is a constant and
with the  constant  $\delta = \delta_0^{-1/3} \in S^1$  we obtain
$\det  \mathcal{F}(\delta \f_0 ) = 1$.
 
 b) Assume  $\det  \mathcal{F}(\f_0 ) = \det  \mathcal{F}(\f_1 ) = 1$. Since $\f_0$ and $\f_1$ are both lifts 
 of $f$ on $\D$, there exists some smooth function $h: \D \rightarrow S^1$ such that 
 $\f_1 = h \f_0$ holds. Then  $ \det  \mathcal{F}(\f_1 ) = \det  \mathcal{F}(h \f_0 ) = 1$
 implies $h^3 = 1$. Hence $h$ is a constant.
 
 The last statement follows from the second one.
 \end{proof}
 
From the Appendix of \cite{DoKoMa} we infer
 
 \begin{theorem} \label{non-compact mLi Legendrian}
 Let $M$ be a non-compact Riemann surface and $f:M \rightarrow \C P^2$ 
 a full minimal Lagrangian  immersion. Then there exists a global lift $\f : M \rightarrow S^5$.
 \end{theorem}
 
 Note, this lift is not necessarily horizontal.
 However we can apply  \cite{DoKoMa} and obtain:
 
 \begin{theorem}
  Let $M$ be a non-compact Riemann surface and $f:M \rightarrow \C P^2$ 
 a full minimal Lagrangian  immersion. Let $\f : M \rightarrow S^5$ denote a global lift
 and  $\tilde{\f} : \D  \rightarrow S^5$ a lift of $\f$ to its universal covering $\D$.
 Moreover, let $\mathcal{F}(\tilde{\f})$ be  the (unique) frame associated with $\tilde{\f}$
 in the sense of Theorem \ref{Thm: normalized coordinate frame} (c). 
 Then the  Gauss maps $\mathcal{G}_j : \D \rightarrow FL_j$ $(j=1, 2, 3)$, defined in section 3.3 of \cite{DoKoMa},  are primitive relative to 
 $\sigma$.
 \end{theorem}
 
  \begin{corollary}
 Let $M$ be a non-compact Riemann surface and $f:M \rightarrow \C P^2$ 
 a full minimal Lagrangian  immersion. Then $f$ can be treated as an immersion without complex points as in \cite{DoKoMa} and thus can be constructed as there.
 \end{corollary}
 
 From this we obtain the following 
 {\bf Construction Principle of all Minimal Lagrangian  Immersions into $\C P^2$}
 defined on a non-compact Riemann surface $M$.
   
{\it First construct a primitive Gauss map relative to $\sigma$ defined on $\D$ which is invariant under the fundamental group of $M$ following the recipe of \cite{DoKoMa}. 

Next pick  the last column of the corresponding extended frame. We obtain a map 
 $\tilde{\f} :\D \rightarrow S^5$. 
This map  projects down to a map $\f:M \rightarrow S^5$, since by construction,  $\tilde{\f} $ is invariant under $\pi_1(M)$. 
 
Finally, an application of the Hopf fibration produces the desired minimal Lagrangian surface in $\C P^2$.}
 
 %%%%%%%%%%%%%%%%%%%%%%%%%%
  \subsection{Horizontal lifts and threefold covers for minimal Lagrangian immersions into \texorpdfstring{$\C P^2$}{}}
  \label{subsection: general case}
 
Recall that we assume that $M$ is different from $S^2$. We use this right below,
  when we state that  $\tilde{f}: \D \rightarrow \C P^2$  has a lift 
  $\tilde{\f}: \D \rightarrow S^5.$
   This is proven by considering the pull back bundle and using that $\D$ is contractible.
   
 \begin{proposition}\label{prop:general}
 Let  $f : M \rightarrow \C P^2$ be a minimal Lagrangian  immersion  and $\tilde{f} : \D \rightarrow \C P^2$ denote the lift $\tilde{f} = f \circ \pi $ of $f$ to the 
 universal cover $\pi : \D \rightarrow M$.
 Then  $\tilde{f} $ has a horizontal lift  $\tilde{\f} : \D \rightarrow S^5 $ and the following statements hold.
 \begin{enumerate}
 \item For $\gamma \in \pi_1(M)$, acting on $\D$ by M\"{o}bius transformations, we obtain that also  $\gamma^*\tilde{\f} $  is a horizontal lift of  $\tilde{f} .$
 
 \item For all $\gamma \in \pi_1 (M)$ we have 
 $(\gamma^*\tilde{\f} ) (z, \bar{z}) = c(\gamma) \tilde{\f} (z, \bar z)$  with the constant $c$ taking values in $S^1$.
 
 \item If $f$ has a global horizontal lift $\f,$ then the homomorphism $c$ is trivial.
 
 \item After  multiplying $\tilde{\f}$ by a scalar multiple in $S^1$ we can assume without loss of generality that 
  $\mathcal{F}(\tilde{\f} )$ is contained in $SU(3)$.
  
  \item For $\tilde{f}$ as just above and $\gamma \in \pi_1 (M)$ we obtain
  \begin{equation} \label{define k}
  \gamma^*( \mathcal{F}(\tilde{\f} ) )(z, \bar z) = c(\gamma) \mathcal{F}(\tilde{\f} ) (z, \bar z) k(\gamma, z ,\bar z),
  \end{equation}
  with $k(\gamma, z ,\bar z) = \di (|\gamma'(z)| / \gamma'(z), |\gamma'(z)| / \overline{\gamma'(z)},1)$, where $\gamma^{\prime}(z) =\gamma_z (z)$.
 \end{enumerate}
  \end{proposition} 
   
 \begin{proof} 
Items $(1)$, $(2)$ and (4) follow easily. Item $(3)$ is a consequence of the identity $f = \pi_H \circ \f,$
which holds by assumption.
It remains to  show the last statement here.
Following P.~Lang's idea in his thesis (for the coordinate frame for CMC immersions  \cite{PLang}), 
we differentiate both sides of $\tilde{\f}(\gamma.z, \overline{\gamma. z})=c(\gamma) \tilde{\f}(z, \bar{z})$, and obtain
$$\tilde{\f}_z (\gamma . z,\overline{\gamma. z}) \gamma^{\prime} = 
c(\gamma) \tilde{\f}_z(z, \bar z).$$
An analogous  result we obtain for  the $\bar z$ derivative. 
Also, since $\gamma$ is an isometry of the induced metric $g=2e^{u} dzd{\bar z}$, we have  equivalently $e^{\gamma^*u} |\gamma^{\prime}|^2 = e^u$.
Putting this together we derive 
$$\gamma^* \mathcal{F}  = c(\gamma) \mathcal{F} k,$$
where  $k(\gamma, z ,\bar z) = \di (|\gamma'(z)| / \gamma'(z),  
|\gamma'(z)| / \overline{\gamma'(z)},1)$, with $\gamma^{\prime}(z) =\gamma_z (z).$
Thus the claim follows. 
 \end{proof}
 
 \begin{corollary}\label{cor:imc}
 Retaining the assumptions above we obtain that 
  $c: \pi_1 (M) \rightarrow S^1$ is a homomorphism of groups.
  Moreover,  we have $c(\gamma)^3=1 \in S^1$
 for all $\gamma \in  \pi_1(M)$. 
 In particular, $c$ is a homomorphism with
  values in the  group 
 $\mathbb{A}_3$ of cubic  
 roots of unity, whence the image of $c$ is either 
 $\{e \}$ or all of $\mathbb{A}_3$.
  \end{corollary}
  \begin{proof}
  Item $(2)$ of the proposition above implies directly the first claim.
  The second claim follows from (\ref{define k}), since we can assume $ \det  \mathcal{F}( \tilde{\f}) = 1$.
  \end{proof}
  
  We also point out
  \begin{corollary} 
  The gauge $k(\gamma,z, \bar z)$ defined in (\ref{define k}) above is a crossed homomorphism, i.e. $k$ satisfies
  \begin{equation} \label{crossed homomorphism}
  k(\gamma \mu, z, \bar z) = k(\mu, z, \bar z) k(\gamma, \mu(z), \overline{ \mu(z)}) 
  \hspace{2mm}  \text{for all} \hspace{2mm} \gamma,\mu \in \pi_1(M).
   \end{equation}
  \end{corollary}

  From this we derive the following
  
  \begin{theorem} \label{describe hatM}
  Let $M$ be a Riemann surface, different from $S^2$, and $f:M \rightarrow \C P^2$ a minimal Lagrangian immersion. 
  Let $\pi: \D \rightarrow M$ denote the universal covering of $M$ and 
  $\tilde{f} = f \circ  \pi : \D \rightarrow \C P^2$ the natural lift of $f$ to $\D$.
  Let $\tilde{\f} : \D \rightarrow S^5 $ denote a horizontal lift of $\tilde{f}$  satisfying 
  $ \det  \mathcal{F}( \tilde{\f}) = 1$. 
  Let $c: \pi_1 (M) \rightarrow S^1$ denote the homomorphism induced by $\tilde{\f}$ and put
  $\Gamma = \ker (c)$. Furthermore, define the Riemann surface   
  $\hat{M} = \Gamma \backslash  \D$.
  Then the following statements hold$:$
  
\begin{enumerate}
 \item[a)] The definitions above induce naturally a sequence of coverings 
  \begin{equation}
 \xymatrix{\D \ar[r]^{\hat{\pi}}&\hat{M}\ar[r]^{\tau}&M.}
  \end{equation}
   Recall that our definitions imply $\pi = \tau \circ \hat{\pi}$.
  Moreover, the covering map $\tau$ has either order $1$ or order $3$.
  
\item[b)] Putting $\hat{f} = f \circ \tau: \hat{M} \rightarrow  \C P^2$ we obtain the commuting diagram
  \[
  \begin{diagram}
    \node{\D}  \arrow{r,l}{\tilde{\f}} \arrow{s,l}{\hat{\pi}} \node{S^5} \arrow{s,r}{\pi_H}  \\
    \node{\hat{M}}\arrow{s,l}{\tau}\arrow{ne,l}{\hat{\f}} \arrow{r,l}{\hat{f}}\node{\C P^2} \\
    \node{M} \arrow{ne,r}{f}
  \end{diagram}
\]
  where  $\hat{\f} : \hat{M} \rightarrow S^5$ is the naturally global horizontal lift of $\hat{f}$.
  
  Then, either $\hat{M} = M$ and $f$ itself has a global lift or $\tau: \hat{M} \rightarrow M$ has order three and $\hat{f}$ has the global lift $\hat{\f} $. Moreover, $\tilde{\f}$ is the common natural lift of $f$ and $\hat{\f}$ to $\D$.
\end{enumerate}
  \end{theorem}
  
  \begin{proof}

It follows from Corollary \ref{cor:imc} that the image of $c$ is either only the identity of $S^1$ or the group $\mathbb{A}_3$ of cubic roots of unity, thus the kernel of $c$ is $\pi_1(M)$ or a subgroup $\Gamma$ satisfying $\mathbb{A}_3\cong \pi_1(M)/\Gamma$. In the first case, $\hat{M}=M$ and $\hat{\f}$ is a global lift of $f$. In the second case, $\hat{M}$ is a threefold covering of $M$ and $\hat{\f}$ is a global lift of $\hat{f}$.
 \end{proof}
 
 \begin{corollary} \label{3-fold cover}
 Let $M$ be a Riemann surface different from $S^2$ and $f: M \rightarrow \C P^2$ a minimal Lagrangian immersion.
 Then either $f$ has a global horizontal lift $\f : M \rightarrow S^5,$  or there exists a threefold covering 
 $\tau: \hat{M} \rightarrow M$ of $M$ such that the minimal Lagrangian immersion 
 $\hat{f} = f \circ \tau: \hat{M} \rightarrow \C P^2$ has a global horizontal lift, while the given 
 $f : M \rightarrow \C P^2$ has no such lift.
In this case we have  $$\pi_1(M) / {\pi_1(\hat{M})} \cong \mathbb{A}_3,$$
 the group of cubic roots of unity, if $\hat{M}\neq M$.
 \end{corollary}

 %%%%%%%%%%%%%%%%%%%%%%%%%%%
\subsection{Threefold quotients of horizontally liftable minimal Lagrangian immersions}
 \label{subsection: compact case}
%%%%%%%%%%%%%%%%%%%%%%%%%%%

This following theorem can be viewed as a converse to the results in the previous two sections.

\begin{theorem} \label{covering to horizontal mLi}
Consider the commuting diagram
  \[
  \begin{diagram}
    \node{\D}  \arrow{r,l}{\tilde{\hat{\f}}} \arrow{s,l}{\hat{\pi}} \node{S^5} \arrow{s,r}{\pi_H}  \\
    \node{\hat{M}}\arrow{s,l}{\tau}\arrow{ne,l}{\hat{\f}} \arrow{r,l}{\hat{f}}\node{\C P^2} \\
    \node{M} 
  \end{diagram}
\]
 Here  $M$ and $\hat{M}$ are  Riemann surfaces and $\D$ is a contractible domain in $\C$.
 Moreover, $\hat{f} : \hat{M} \rightarrow \C P^2$  is a minimal Lagrangian immersion
 with (global) horizontal lift $\hat{\f} : \hat{M} \rightarrow S^5$
 and natural lift $\tilde{\hat{f}} : \D \rightarrow S^5.$ 
 We assume that
 $ \xymatrix{\D \ar[r]^{\hat{\pi}}&\hat{M}\ar[r]^{\tau}&M}$ is a sequence of coverings 
 and that $\tau$ has order three.
 
 Finally, we assume that for the action of $\pi_1(M)$ on $\D$ we have 
  $\tilde{\hat{\f}}(\gamma.z) = q(\gamma) \tilde{\hat{\f}}(z) $  for some  
  homomorphism $q: \pi_1(M) \rightarrow S^1$ and all $\gamma \in \pi_1(M)$
  and that $\pi_1(\hat{M})$ is a normal subgroup of $\pi_1(M).$
 Then there exists a minimal Lagrangian immersion $f: M \rightarrow \C P^2$ 
 satisfying $ f \circ \tau = \hat{f}$.
 Moreover,  the natural lifts $\tilde{f}$ and  $\tilde{\hat{\f}}$ are equal.
  \end{theorem}
 
 \begin{proof}
 From item $(3)$ of Proposition 3.5 we know that $q(\gamma) = 1$ for $\gamma \in \pi_1(\hat{M}).$
 Hence $q$ descends to a homomorphism $q_0:  \pi_1(M) / {\pi_1(\hat{M})} \cong \mathbb{A}_3.$
 Let $a_0$ be a generator of  $ \pi_1(M) / {\pi_1(\hat{M})} $ and $a^*_0 = q_0(a_0)  \in S^1.$  
 Then $\tilde{\hat{\f}}(a_0.z) = q_0^* {\tilde{\hat{\f}}}(z)$ for $z \in \D$.  As a consequence, 
from $  \pi_H \circ \tilde{\hat{\f} } = \hat{f} \circ \hat{\pi}$ we obtain by projection
 $\hat{\f}(a_0^*.z) = [q_0]^* \hat{\f}(z)$, where $a_0^*$ is the image of the generator $a_0$ to an automorphism of $\hat{M}$, which makes sense, since $\pi_1(M)$ normalizes $\pi_1(\hat{M})$ .
 Thus there exists a map $f: M \rightarrow \C P^2$ satisfying  $ f \circ \tau = \hat{f}$.
 The remaining claims can be verified by a straightforward argument.
 \end{proof}

 By the discussion above  it is clear that to obtain a minimal Lagrangian immersion 
 from $M$ into $\C P^2$  with a global horizontal lift the fundamental group of $M$ needs to be equal to the fundamental group of 
 $\hat{M}$ or it needs to contain one more generator ``$\kappa$", the cube of which is contained in 
 $\pi_1(\hat{M}).$  The latter situation can be characterized as follows:
 
\begin{theorem} \label{degree covering to mLi}
Assume we have the situation as in Theorem  \ref{covering to horizontal mLi}, in particular, 
we assume $M \neq \hat{M}$. Let $\kappa$ be a generator of $\pi_1(M)$ which acts on $\D$ as 
a non-trivial symmetry of $\tilde{f},$ normalizes  $ \pi_1(\hat{M})$ and has order three as an 
automorphism of $\hat{M}$.
Then the minimal Legendrian immersion $\hat{\f} : \hat{M} \rightarrow S^5$ induces a 
minimal Lagrangian immersion $f:M \rightarrow \C P^2$ if and only if the additional generator 
$\kappa$ of $\pi_1(M)$
 maps under the monodromy  representation of $\hat{f}$ to a multiple of the identity matrix, $cI$.
 
 In this case $M$ is  the quotient of $\hat{M}$ induced by the symmetry  $\kappa$  of order three
 induced by $\kappa$. In particular,  $\hat{M}$ is a threefold cover of $M$.
\end{theorem}
\begin{proof}
The ``if'' part follows easily from the diagram in Theorem  \ref{describe hatM}.
Conversely, put $\hat{f} = \pi_H \circ \hat{\f}.$ Then $\kappa^* \hat{f} = \hat{f}$ and $\hat{f}$ descends to a map
$f : M \rightarrow \C P^2$, proving the claim.
\end{proof}

 %%%%%%%%%%%%%%%%
\section{Symmetries for minimal Lagrangian immersions from arbitrary Riemann surfaces}
\label{sec:4}
%%%%%%%%%%%%%%%%%%%%%%%%%

\subsection{Basic definition of symmetry}
In \cite{DoMaB1} we studied symmetries for minimal Lagrangian immersions from contractible Riemann surfaces. Now
let $M$ be an arbitrary Riemann surface different from $S^2$.
Let $f: M \rightarrow \C P^2$ be an immersion.
A central topic in this paper will be the notion of a ``symmetry of $f$''.
While a basic definition of a symmetry $\mathcal{R}$ for $f$ may  only be 
$$\mathcal{R} f(M) = f(M),$$ 
 for $\mathcal{R}\in \mathrm{Iso}_0(\mathbb{C}P^2)=PSU(3)$, it is certainly  helpful to have additional information. The idea how this could be
obtained will follow from

\begin{theorem}\label{thm:symmetry}
Let $M$ be a   Riemann surface and $\pi: \D\rightarrow M$ its universal cover
and  $f:M \rightarrow \C P^2$  an immersion. We assume
\begin{itemize}
\item[a)] The metric induced by $f$ on $M$ is complete,
\item[b)] $\mathcal{R} f(M) = f(M)$ for some  $\mathcal{R} \in  \mathrm{Iso}_0(\C P^2)$.
\end{itemize}
Then there exists some $\tilde{\gamma}  \in \mathrm{Aut}(\D)$ such that for the natural lift
$\tilde{f} = f \circ \pi$ we have
\begin{equation}
\tilde{f}(\tilde{\gamma}. z) = \mathcal{R} \tilde{f}(z) \hspace{2mm} \mbox{for all} \hspace{2mm} z \in \D.
\end{equation}
Moreover, 
$\tilde{\gamma}$ normalizes $\pi_1(M)$, the group  of deck transformations corresponding to $M$ on $\D$,
 therefore there exists an  automorphism $\gamma \in \mathrm{Aut}(M)$ such that
\begin{equation}
f(\gamma.z) = \mathcal{R} f(z)  \hspace{2mm} \mbox{for all} \hspace{2mm} z \in M.
\end{equation}
\end{theorem}

\begin{proof}
Let $\tilde{f} : \D \rightarrow \C P^2$ denote the lift of $f$ to the universal cover 
$\pi: \D \rightarrow M$ of $M$.

It is easy to see that $\tilde{f}$ satisfies the conditions a) and b) as well. Therefore, by the argument given in Section 4.1,
\cite{DoMaB1},
there exists some automorphism  $\tilde{\gamma}$ of $\D$
such that $\tilde{f}(\tilde{\gamma}.z)  = \mathcal{R} \tilde{f}(z)$ for all $z\in \D$.

Moreover, 
$\pi_1(M)$ is normalized by $\tilde{\gamma}$, 
therefore $\tilde{\gamma}$ 
descends to an automorphism of $M$ and the last claim follows by a simple calculation.
\end{proof} 
 
 In view of the theorem above,  in this paper, for an immersion $f: M\rightarrow \C P^2$,  a \emph{symmetry} will always be a pair $(\gamma, \mathcal{R}) \in
\mathrm{Aut}(M)\times \mathrm{Iso}_0(\C P^2)$, such that 
\begin{equation} \label{basic-symmetry}
f(\gamma . z) = \mathcal{R}  f(z) \hspace{2mm} \mbox{for all} \hspace{2mm} z \in M
\end{equation}
 holds.
 
 An analogous definition will apply to immersions $\f : M \rightarrow S^5$.

 From here on we will always assume 
 that $f: M \rightarrow \mathbb{C}P^2$ is  \emph{full},  i.e.  if we have some  $\mathcal{R} \in  \mathrm{Iso}_0(\C P^2)$ such that
$\mathcal{R}f(z) = f(z)$ for all $z \in M$, then $\mathcal{R}  = id$.

The following result follows by a straightforward argument.
\begin{lemma}
Let $M$ be a Riemann surface different from $S^2$. Let  $f: M\rightarrow \C P^2$ be a full immersion and
$(\gamma, \mathcal{R}) \in
\mathrm{Aut}(M)\times \mathrm{Iso}_0(\C P^2)$ a symmetry of $f$.
Then $\mathcal{R}$ is uniquely determined by $\gamma$ and $f$.
Moreover, if $G$ is a group of automorphisms of $M$ such that for each $\gamma \in G$ there exists some $\mathcal{R}_\gamma$ such that $(\gamma, \mathcal{R}_\gamma )$ is a symmetry for $f$, then the map $\gamma \mapsto \mathcal{R}_\gamma$  is well defined and a homomorphism of groups.
\end{lemma}

Now, as already carried out in a proof above, let's lift $f$ to the universal cover $\D$ with projection $\pi : \D \rightarrow M$, by putting  $\tilde{f}: \D \rightarrow \mathbb{C}P^2$,  where $\tilde{f} = f \circ \pi$.
It is well  known that for each symmetry $(\gamma, \mathcal{R})$ of $f$  one finds on the universal cover 
$\D$  some automorphism $\tilde{\gamma}$  such that 
$\pi \circ \tilde{\gamma} = \gamma \circ  \pi$ and moreover
\begin{equation}\label{eq:sym_tilde_f}
\tilde{f}(\tilde{\gamma}. z) = \mathcal{R} \tilde{f}(z) \hspace{2mm} \mbox{for all} \hspace{2mm} z \in \D.
\end{equation}

\begin{lemma}
Let's keep the assumptions and the notation introduced above. If  $f: M \rightarrow \mathbb{C}P^2$ is full, then also  $\tilde{f}: \D \rightarrow \mathbb{C}P^2$ is full.
Moreover, if $\Gamma$ is a group of symmetries of $f$, then 
$\tilde{\Gamma} = \{ (\tilde{\gamma}, \mathcal{R}) \in \mathrm{Aut}(\D)\times 
 \mathrm{Iso}_0(\mathbb{C}P^2) | \pi \circ \tilde{\gamma} = \gamma \circ \pi \text{ and }
  (\gamma, \mathcal{R})\in \Gamma \}$ 
  is a group of symmetries of  $\tilde{f}$.
\end{lemma}

As a special case, we have
\begin{corollary} \label{decktrafo}
We retain the assumptions and the notation of the previous lemmas.
Using the group of symmetries 
$\Gamma = \pi_1(M) \times \{I\}$ of $M$, where $\pi_1(M)$ acts trivially on $M$, we infer that $\tilde{\Gamma} \cong \pi_1(M) \times \{I\}$ is a group of symmetries of $\tilde{f}$, where here $\pi_1(M)$ acts on $\D$ as the group of deck transformations of $M$.
\end{corollary}

%%%%%%%%%%%%%%%%%%%%%%%%%%
\subsection{The action of symmetries on horizontal lifts and
gauged frames}
\label{subsection: action of symmetries}

From now on, we restrict our attention to  a full minimal Lagrangian immersion $f:M \rightarrow \C P^2$  into $\C P^2$ from a Riemann surface $M$ different from $S^2$. 
Such immersions do not have any complex points (see \cite{DoMaB1}).
Thus the natural lift $\tilde{f}: \D\rightarrow \C P^2$ of $f$ to the universal cover $\D$ of $M$, as introduced above, 
also is a full minimal Lagrangian immersion. It follows from Corollary \ref{Cor:globle-lift} that $\tilde{f}$ 
is horizontally liftable  with its horizontal (Legendrian) lift $\tilde{\f}$. 

Assume that we have a group of symmetries of type $(\gamma, \mathcal{R})$ of $f$ 
such that
\begin{equation}
f(\gamma.z) = \mathcal{R} f(z), \quad z \in M.
\end{equation}

Note that  $\mathcal{R} =  \mathcal{R}_\gamma$ is uniquely determined, if $f$ is full.
As a consequence, the map $\gamma \mapsto \mathcal{R}$ is a homomorphism of groups.
From here on we will frequently write $[R]$ for the isometry of $\C P^2$ naturally induced
by some $R \in SU(3)$. 
For the horizontal lift $\tilde{\f}$   of the natural lift  $\tilde{f}$    we obtain

\begin{lemma}
Let $f$ be a full  minimal Lagrangian surface  and  $\tilde{\f}$  a
 horizontal lift of its natural lift $\tilde{f}$ to the universal cover $\D$.
If $(\gamma, \mathcal{R})$  is a symmetry of $f$, then
$\tilde{\f} (\tilde{\gamma}.z) = \mathring{c}_{\gamma} R_{\gamma} \tilde{\f}(z)$ for some $R_{\gamma} \in SU(3)$ satisfying $[R_{\gamma}] = \mathcal{R}$, some  $\mathring{c}_{\gamma} \in S^1$ and all $z\in \D$.
\end{lemma}

\begin{proof} From the above discussion, there exists a symmetry $(\tilde{\gamma}, \mathcal{R})$ of $\tilde{f}$ such that \eqref{eq:sym_tilde_f} holds. 
 It is easy to verify that $\tilde{\f}(\tilde{\gamma}.z)$ is a horizontal lift of $\tilde{f}(\tilde{\gamma}.z)$ and 
$R_{\gamma}\tilde{\f}(z)$ is a horizontal lift of $\mathcal{R}\tilde{f}(z)$, where $\mathcal{R} = [R_{\gamma}]$, $R_{\gamma} \in SU(3)$. 
It follows from \eqref{eq:sym_tilde_f} that $\tilde{\f}(\tilde{\gamma}.z)$ and $R_{\gamma} \tilde{\f}(z)$ 
only differ by a constant  $\mathring{c}_{\gamma} \in S^1$. 
Hence $\tilde{\f}(\tilde{\gamma}.z) = \mathring{c}_{\gamma} R_{\gamma}\tilde{\f}(z)$ and the  claim follows.
\end{proof}

In Sections \ref{subsec:frames} and \ref{subsec:2.3},
we have defined the coordinate frame $\mathcal{F}$  and
the gauged frame $\F$ of some horizontal lift of any minimal Lagrangian immersion defined on a contractible domain $\D$.
In the present situation we define the frame
$\F$ for $\f$ by  
the frame for $\tilde{\f}$.

In Proposition \ref{prop:frame},
we have proven, that the frame for $\tilde{\f}$
is uniquely determined, if we assume that it is contained in $SU(3)$ and attains the value $I$ at some fixed base point $z_0$.
 We thus assume from here on that  the (gauged coordinate ) frame $\F$ for $\tilde{\f}$ is 
 contained in $SU(3)$ and attains the value $I$ at the fixed base point $z_0$.

Then we  have
 %that the global horizontal (Legendrian) lift $\f$ of liftable  $f$ is chosen such that the frame $\F$ for $\tilde{\f}$  is contained in $SU(3)$ and attains the value $I$ at some fixed base point $z_0$.
%(From Proposition 4.5 in \cite{DoMaB1}, we have)
\begin{theorem} \label{trafoIF} We use the notation and the definitions as in the previous lemma and just above.
If $\F: \D \rightarrow SU(3)$ denotes the gauged frame of $\tilde{\f}$ satisfying
$\F(z_0, \overline{z_0}) = I$, then 
\begin{equation}\label{eq:define R}
\F(\gamma.z, \overline{\gamma.z}) = c_{\gamma}R_{\gamma} \F(z, \bar{z}) k(\gamma, z,\bar{z})
\end{equation}
 for the same $R_{\gamma} \in SU(3)$ as in the last lemma,
and 
$k(\gamma, z ,\bar z) = 
\di (|\gamma'| / \gamma', |\gamma'| / \bar{\gamma}',1) \in K \cong U(1)$, where $\gamma^{\prime}=\gamma_z$.

Moreover, $c_{\gamma} = \mathring{c}_{\gamma}^{\frac{1}{3}}$ with $\mathring{c}_{\gamma}$ as in the lemma above 
and 
the map 
 ($\gamma\mapsto c_{\gamma}R_{\gamma}$) is well defined and a homomorphism of groups and 
$ k(\gamma, z,\bar{z})$ is a crossed homomorphism.
\end{theorem}

\begin{proof}
The proof for the expression for $k$ is verbatim the same as the one for  $(\ref{define k})$.  
Applying the definition of the gauged coordinate frame $\F(\tilde\f)$  
and the result above for $\tilde{\f}$ implies the claim.
\end{proof}

%%%%%%%%%%%%%%%%%%%%%%%%%
\subsection{Symmetries for the loopified  quantities}

The theorem above  yields for the extended frame 
$\F(z, \bar{z},\lambda)= \F_\lambda (z, \bar{z}) \in \Lambda SU(3)_\sigma$, satisfying 
$\F (z_0, \bar{z}_0,\lambda) = I,$ and the equation
\begin{equation} \label{transF}
\F(\gamma.z, \overline{\gamma.z},\lambda) = \chi(\gamma,\lambda) \F(z, \bar{z},\lambda) 
k(\gamma, z, \bar{z}),
\end{equation}
 with $\chi\in \Lambda SU(3)_{\sigma}$.
Moreover, 

\begin{corollary}
Retaining the assumptions and the notation of equation \eqref{transF} above we obtain:

The map $\gamma \rightarrow \chi(\gamma,\lambda)$ is a homomorphism of groups satisfying 
$\chi(\gamma, \lambda = 1) = c_{\gamma} R_{\gamma}$.
Furthermore, the map $\gamma \rightarrow k(\gamma, \cdot)$ is a crossed homomorphism.
\end{corollary}

%%%%%%%%%%%%%%%%%%%%%%%%%%%%%%%%
\subsection{Normalized frames, holomorphic frames and their transformation behaviour}
%%%%%%%%%%%%%%%%%%%%%%%%%%%%

As discussed in detail in Section 2.3 of \cite{DoMaB1} the loop group method associated with an extended frame uses two more special types of frames.

(1) \textbf{The normalized extended frame}. 

\begin{proposition}
Let $\F$ denote the extended frame of a full minimal Lagrangian
immersion  $f: M \rightarrow \C P^2$. Then there exists a discrete subset $\mathcal{S} \subset \D$ such that for all $ z \in \D\backslash \mathcal{S} $ the extended frame can be decomposed in the form
\begin{equation}
\F (z, \bar z, \lambda) = \F_-(z,\lambda) \F_+(z, \bar z, \lambda)
\end{equation}
with $\F_{-}\in \Lambda^{-}_* SL(3,\C)_\sigma$, $\F_{+} \in \Lambda^{+} SL(3,\C)_\sigma$.
Moreover, $\F_-$ is meromorphic in $z$ and we have 
\begin{equation}
\eta_-(z, \lambda) = \F_-(z,\lambda)^{-1} d\F_-(z, \lambda ) =
 \lambda^{-1} \eta_{-1} (z) dz. 
 \end{equation}
\end{proposition}

The meromorphic matrix function $\F_-(z, \lambda)$ introduced above will be called 
\emph{normalized extended frame} associated with $f$ (or $\f$).
The meromorphic differential one-form $\eta_-$ introduced above will be called the
\emph{normalized potential} associated with $f$ (or $\f$).

From equation (\ref{transF}) one derives immediately
\begin{equation} \label{transFminus}
\F_-(\gamma.z, \lambda) = \chi(\gamma,\lambda) \F_-(z, \lambda) V_+(\gamma, z, \bar z, \lambda)
\end{equation}
with  $V_+ \in \Lambda^{+} SL(3,\C)_\sigma$.

Finally one also obtains a transformation formula for the normalized potential
\begin{equation} \label{transetaminus}
\eta_-(\gamma.z, \lambda) =  V_+(\gamma, z, \bar z, \lambda)^{-1} \eta_-(z, \lambda) V_+(\gamma, z, \bar z, \lambda) + V_+(\gamma, z, \bar z, \lambda)^{-1} dV_+(\gamma, z, \bar z, \lambda)
\end{equation}
with  $V_+$ as above.

%%%%%%%%%%%%%%%%%%%%%%%%
(2) \textbf{Holomorphic  extended frame.} 
 For a non-compact Riemann surface $M$ one can  avoid having to deal with meromorphic  matrix functions by  using \emph{holomorphic  extended frames}.

In this case one writes 
\begin{equation} \label{from F to C} 
\F(z, \bar z, \lambda) = C(z,\lambda) W_+(z, \bar z, \lambda)
\end{equation}
with $C\in \Lambda SL(3,\C)_{\sigma}$ and $W_+\in \Lambda^+SL(3,\C)_{\sigma}$
by solving a $\bar{\partial}$-equation  (verbatim as in \cite{DPW}).

Here $C$ is holomorphic in $ z \in  \mathbb{D}$  and we obtain
\begin{equation}
\eta^h_-(z, \lambda) = C(z,\lambda)^{-1} dC(z, \lambda ) = \sum_{j=-1}^{\infty}
 \lambda^j \eta^h_j (z) dz \hspace{2mm} \mbox{ with} \hspace{2mm} \eta^h_j (z) \in \mathfrak{g}^\C_{j\,{\mathrm{mod}}\, 6}.
\end{equation}
Note that $C(z, \lambda)$ is not uniquely determined.
The holomorphic matrix function $C(z, \lambda)$ introduced above is 
called a  \emph{holomorphic extended frame} associated with $f$ or $\f$.
The holomorphic  differential one-form $\eta^h_-$ introduced above is
called  \emph{a holomorphic potential} associated with $f$ (or $\f$).

 We will see in Theorem \ref{thm:cpt}  below  that for a compact $M$ one can find a meromorphic $C$ satisfying equation (\ref{from F to C}). 
For the following arguments we will only assume that we have any 
meromorphic $C$ satisfying equation (\ref{from F to C}).
The transformation formulas are almost verbatim as just above. From equation (\ref{transF}) one obtains
\begin{equation} \label{transC}
C(\gamma.z, \lambda) = \chi(\gamma,\lambda) C(z, \lambda) Q_+(\gamma, z, \bar z, \lambda)
\end{equation}
with  $Q_+ \in \Lambda^{+} SL(3,\C)_\sigma$
and $\chi(\gamma,\lambda) \in \Lambda SU(3)_{\sigma}$. Hence we obtain a transformation formula for the holomorphic potential
\begin{equation} \label{transeta}
\eta^h_-(\gamma.z, \lambda) =  Q_+(\gamma, z, \bar z, \lambda)^{-1} \eta^h_-(z, \lambda) Q_+(\gamma, z, \bar z, \lambda) + Q_+(\gamma, z, \bar z, \lambda)^{-1} dQ_+(\gamma, z, \bar z, \lambda)
\end{equation}
with  $Q_+$ as above.

%%%%%%%%%%%%%%%%%%%%%%%%%%%%%%%
\subsection{Constructing general minimal Lagrangian surfaces}
\label{subsec:construction}
%%%%%%%%%%%%%%%%%%%%%%%%%%%%%%%

 In this subsection we will outline how one can construct any minimal Lagrangian immersion from some Riemann surface $M \neq S^2$ to $\C P^2$.
 
 This outline is given for the most general case of ``potentials".
 It thus is fairly technical. For most applications a much simpler and easier 
 choice of potentials is possible. More precisely, in most cases we can
 assume without  loss of generality that the crossed homomorphism $Q_+$ occurring below is equal to $I$.
 For details on this see Section \ref{sec:inv frame}.
 
{\it Step 1:}  Let $M \neq S^2$ be a Riemann surface and $\pi: \D \rightarrow M$ its universal cover. Recall, in this case $\D$ is with loss of generality the complex plane 
 $\C$ or the unit disk $\mathcal{D}$.  Recall that the fundamental group 
 $\pi_1(M)$ of $M$ acts on $\D$ by M\"{o}bius transformations. Let $\eta$ be a meromorphic differental 1-form defined on $\D$,
 \begin{equation}
 \eta(z,\lambda) = \sum_{j = -1}^\infty  \lambda^j \eta_j(z),
 \end{equation}
 where $\lambda \in \C^*$ and $\eta_j(z) \in \mathfrak{g}_j.$
 
 In view of the results of the last subsection we require for all $\gamma \in \pi_1(M)$, 
$ \eta$ satisfies equation (\ref{transeta}) with 
$Q_+(\gamma, z,\lambda) \in \Lambda SL^+(3,\C)_{\sigma}$ .
In addition, we require
 \begin{proposition}
 $Q_+$ satisfies the crossed homomorphism property
 \begin{equation}
      Q_+(\gamma \mu, z, \bar z) = Q_+(\mu, z, \bar z) Q_+(\gamma, \mu(z), \overline{ \mu(z)}) 
  \hspace{2mm}  \text{for all} \hspace{2mm} \gamma,\mu \in \pi_1(M).
 \end{equation}
 \end{proposition}
 
 {\it Step 2:} Solve the initial value problem $dC = C \eta,$ $ C(z_0,\lambda)= I$
 for some fixed base point $z_0$ which is not a pole of $\eta$ and assume that this ode has a meromorphic solution.
 \vspace{4mm}
 
 It is easy to see that
 $$ \gamma^*C(z,\lambda) = \rho(\gamma,\lambda) C(z, \lambda) Q_+(\gamma,z, \lambda)$$
 holds.
Note, since $Q_+$ satisfies the crossed homomorphism property  we obtain

\begin{proposition}
The map $\rho: \pi_1(M) \rightarrow  \Lambda SL(3,\C)_{\sigma}$ is a homomorphism of groups.
\end{proposition}
 We require for all $\gamma \in \pi_1(M)$ and all $\lambda \in S^1$: 
 $$ \rho(\gamma,\lambda) \in  \Lambda SU(3)_{\sigma}.$$
 
{\it Step 3:} Perform the unique Iwasawa decomposition 
$$ C = \F L_+.$$

\begin{theorem}
 Under the assumptions and the notation introduced above the following statements hold for all $\gamma \in \pi_1(M)$ and all $\lambda \in S^1$:
 \begin{enumerate}
     \item $\tilde{\F}(\gamma.z, \overline{\gamma.z}, \lambda) = \rho(\gamma,\lambda)  \tilde{\F}(z, \bar z, \lambda) k(\gamma,z)$
     for all $z\in \D$ and some diagonal $k(\gamma,z) \in SU(3).$
     \item Let $\tilde{\f}$ denote the last column of $\tilde{\F}.$
      Then we obtain $\tilde{\f}: \D \rightarrow S^5$ and 
      $$\tilde{\f}(\gamma.z, \overline{\gamma.z}, \lambda) = \rho(\gamma,\lambda) \tilde{\f}(z, \bar{z})$$ for all $z\in \D$, $\lambda \in S^1.$ 
 \end{enumerate}
 \end{theorem}
 \vspace{4mm}
 
 {\it Step 4:} In order to descend $\tilde{f}$ to a non-simply connected Riemann surface (say for $\lambda=1$), we require:
 For $\lambda = 1$  we obtain  
 $$\rho(\gamma, \lambda = 1) =  c(\gamma) I.$$
 
 Note that this implies $$ c(\gamma)^3 = 1.$$

As a consequence, assuming finally $ c \equiv 1$ 
 we can descend $\tilde{\f}$ to a map $$f : \hat{M} \rightarrow S^5 \rightarrow \C P^2.$$
 If $ \hat{\Pi} = Ker(c) = \pi_1(M)$, then $\hat{M} = M$. Otherwise $\hat{M}$ is a threefold cover of $M$.
 
 \begin{theorem}\label{thm:cons-mli}
The construction outlined above yields a minimal Lagrangian immersion 
 $f : M \rightarrow \C P^2$ and each minimal Lagrangian immersion of this type can be obtained this way.
 \end{theorem}

 %%%%%%%%%%%%%%%%%%%%%%%%%%%%%%%
\section{Invariant frames and potentials for groups of Deck transformations}
\label{sec:inv frame}

The transformation formulas stated in the last subsection 
are not easy to handle.
Fortunately, in the context of constructing minimal Lagrangian immersions on some general Riemann surface $M$ (different from $S^2$) one can simplify the situation considerably.

It seems to be useful to split the cases $M$ \emph{compact} and $M$  \emph{non-compact} into different subsections.

We consider a horizontally liftable minimal Lagrangian
immersion    $f: M \rightarrow \C P^2$,  its horizontal lift   
$\f:M \rightarrow  S^5$ and its gauged extended frame 
$\F:\mathbb{D} \rightarrow \Lambda SU(3)_\sigma$ defined on the 
universal covering  $\mathbb{D}$ of $M$.  

We recall from Theorem \ref{trafoIF} that in this case we can assume without loss of generality that
$\F (z_0, \bar{z}_0) = I$ holds.
We distinguish two cases:

%%%%%%%%%%%%%%%%%%%%%%%%%%%
\subsection{Invariant frames and potentials in the case where $M$ is non-compact:
\newline Case \texorpdfstring{$\D = \C$}{}: }\label{subsec:non-compact-C}
%%%%%%%%%%%%%%%%%%%%%%%%%%

In this case we have  either $ M =\C$, $M = {\C}^*$ or $M = \C / {\mathbb{Z}}$, a cylinder.
Moreover,  the fundamental group of $M$ acts on the universal cover $\C$ of $M$ by a discrete groups of translations. In particular, $\pi_1(M)=p\mathbb{Z}$ is abelian.  We thus obtain

\begin{theorem}
If $M$ is a non-compact Riemann surface with universal cover $\D = \C$, then the coordinate
frame $\F : \D \rightarrow SU(3)$ of some horizontally liftable minimal Lagrangian immersion $f: M \rightarrow \C P^2$ satisfies
$$\gamma^* \mathbb{F} = c_\gamma \mathbb{F},$$
where $\gamma$ is a translation.
\end{theorem}

\begin{proof}
In view of $\gamma^{\prime}=0$, it suffices to recall the formula for $k$ in Proposition \ref{prop:general}.
\end{proof}

From this it follows immediately
\begin{theorem}
If $M$ is a non-compact Riemann surface with universal cover $\D = \C$, then, 
after introducing the loop parameter in the usual way,  the extended 
frame $\F_{\lambda}$ of some liftable minimal Lagrangian immersion 
$f: M \rightarrow \C P^2$ satisfies
$$\gamma^* \mathbb{F}_\lambda = 
\chi(\gamma,\lambda) \mathbb{F}_\lambda$$
for all $\gamma \in \pi_1(M)$.
\end{theorem}
Since $\mathbb{F}$ and $\mathbb{F}_\lambda $ attain the value $I$ at the base point $z_0$,
for $\lambda =1$ we reproduce the last theorem.

%%%%%%%%%%%%%%%%%%%%%%%%%%%
\subsection{Invariant frames  and potentials in the case where $M$ is non-compact: 
Case \texorpdfstring{$\D = \mathcal{D} = \mathfrak{H}$}{}: }\label{subsec:non-compact-H}
%%%%%%%%%%%%%%%%%%%%%%%%%%

We have  considered so far the only possible universal covers $\C$ and $\mathcal{D}$, the unit disk. But in this section it is sometimes more convenient to
replace $\mathcal{D}$ by the biholomorphically equivalent domain $\mathfrak{H}$, the upper half-plane.
In this case the group of biholomorphic automorphisms is 
$Aut(\mathfrak{H}) = SL(2, \R)$.

In the previous subsection we have seen that the coordinate frame (and the associated extended frame) have a transformation behaviour, where
the crossed homomorphism for $k,$ as  stated in Proposition \ref{prop:general} does not show up.
The formula for $k$ stated in Proposition \ref{prop:general} indicates that in the present case the formula for $k$ will not be $I$
for general $\gamma \in \pi_1(M)$.

To obtain a formula similar to the case considered in the previous section (where $k$ is $I$)  we want to find an extended frame
(different from the coordinate frame, if necessary), where the crossed homomorphism 
$k: \pi_1(M)\times \mathfrak{H}\rightarrow K$, stated explicitly  
in  \eqref{define k} and \eqref{eq:define R}, does not show up.

We recall, that $k(\gamma,z, \bar z)$ involves $\gamma'(z)$.
In our present setting it turns out that all $\gamma \in \pi_1(M)$ have the form 
$$  \gamma(z) = \frac{az + b}{cz + d},$$
with the coefficient matrix in $SL(2,\R)$.

Putting 
\begin{equation}
    j(\gamma,z) = cz + d,
\end{equation}
we derive
\begin{equation}
    \gamma^\prime (z) = j(\gamma, z) ^{-2}.
\end{equation}

%%%%%%%%%%%%%%%%%%%%%%%%%%%%%%%%
Before we continue, we would like to spell out explicitly in our present notation,
 and differently from  \eqref{define k} and \eqref{eq:define R},
how $k(\gamma, z, \bar z)$ looks like in terms of $j(\gamma, z).$

\begin{proposition} \label{exolicitk}
Let  $\gamma(z) = \frac{az + b}{cz + d}$ be a Deck transformation of $\pi: \D\rightarrow M$ and $  j(\gamma,z) = cz + d$. Then $ \gamma^\prime (z) = j(\gamma, z) ^{-2}$ and we obtain
\begin{equation}\label{eq:k}
k(\gamma, z, \bar z )=\mathrm{diag}(\frac{j(\gamma, z)}{\overline{j(\gamma, z)}}, \frac{\overline{j(\gamma, z)}}{j(\gamma, z)}, 1).
\end{equation}

\end{proposition}
\begin{proof}

Substituting the expression of $\gamma$ in the proof of Proposition \ref{prop:general}

we derive 
$$\gamma^* \mathbb{F} k^{-1} = c_\gamma \mathbb{F},$$
where $k=\mathrm{diag}(\sqrt{\frac{\bar{\gamma}^{\prime}}{\gamma^{\prime}}},\sqrt{\frac{\gamma^{\prime}}{\bar{\gamma}^{\prime}}},1)$.
Since, for a M\"{o}bius transformation $\gamma$, we obtain $\gamma^{\prime}(z)=j(\gamma, z)^{-2}$ with $j(\gamma, z)=cz+d$, the claim follows. 
\end{proof}

Consider next $\pi: \D\rightarrow \pi_1(M) \backslash {\D}$. 
Since  $j(\gamma, z)$ is a crossed homomorphism with values in $\C^*$, we can apply Corollary 30.5 and Exercise 31.1 of \cite{Forster} and infer that there exists  a holomorphic function $h:\D\rightarrow \C^*$ such that
\begin{equation}\label{eq:j-spliting}
j(\gamma, z)=h(z)h(\gamma. z)^{-1}.
\end{equation}

Substituting into  \eqref{eq:k}, we derive

\begin{theorem}
Assume that $M$ is non-compact. For any frame $\mathbb{F}$ satisfying (\ref{transF}) 
there exists  a holomorphic function $h:\D\rightarrow \C^*$ and a real analytic function $p$ on $\D$ such that 
$k(\gamma, z, \bar z)=p(z, \bar{z})p(\gamma . z, \overline{\gamma . z})^{-1}$, where 
$p(z, \bar{z})=\mathrm{diag}(\frac{h(z)}{\overline{h(z)}}, \frac{\overline{h(z)}}{h(z)}, 1)$ and $j(\gamma, z)=h(z)h(\gamma.z)^{-1}$. 
\end{theorem}

\begin{corollary}\label{cor5.3}
Using the notation above, we set $\hat{\mathbb{F}}(z, \bar{z}) = \mathbb{F}(z, \bar{z}) p(z, \bar{z})$.
Then  
for all $\gamma \in \pi_1(M)$,
\begin{equation}
\gamma^* \hat{\mathbb{F}} = c_\gamma \hat{\mathbb{F}}.
\end{equation}
\end{corollary}

\begin{corollary}
Using the assumptions and the notation above, we introduce the loop parameter as usual. Then the extended frame $\hat{\mathbb{F}}_\lambda$ satisfies
\begin{equation}
\gamma^* \hat{\mathbb{F}}_\lambda = \chi(\gamma, \lambda)\hat{\mathbb{F}}_\lambda.
\end{equation}
\end{corollary}

The property just listed can also be obtained for some holomorphic extended frame.
This will be very helpful for finding a proper potential if one wants to apply the
loop group  method to the construction of minimal Lagrangian
immersions into $\C P^2$.

Recall from subsection \ref{subsection:loop method} that there always exists some holomorphic matrix function 
\begin{equation*}
C_0: \mathbb{D}\rightarrow \Lambda SL(3,\mathbb{C})_{\sigma},
\end{equation*}
such that 
\begin{equation*}
\F=C_0V_+,
\end{equation*}
with some real analytic matrix $V_+: \mathbb{D}\rightarrow \Lambda^{+} SL(3,\mathbb{C})_{\sigma}$ 
and $C_0(z_0,\lambda)=I$.
Thus 

\begin{eqnarray}
C_0(\gamma . z,\lambda)&=&\F(\gamma . z, \overline{\gamma . z}, \lambda) 
V_+(\gamma. z, \overline{\gamma . z}, \lambda)^{-1}\nonumber \\
&=&\chi(\gamma, \lambda) C_0(z,\lambda) V_+(z, \bar{z},\lambda) k(\gamma, z, \bar{z})
V_{+}(\gamma . z, \overline{\gamma . z}, \lambda)^{-1}. \label{eq:C}
\end{eqnarray}

Set 
\begin{equation*}
W_+(\gamma, z, \lambda):=V_+(z, \bar{z},\lambda) k(\gamma, z, \bar{z}) 
V_+(\gamma . z, \overline{ \gamma . z}, \lambda)^{-1}.
\end{equation*}
Then an easy computation shows
\begin{equation*}
W_+(\gamma \mu, z, \lambda)=W_+(\mu, z,\lambda) W_+(\gamma, \mu. z, \lambda),
\end{equation*}
that is, $W_+(\gamma, z, \lambda)$ is  a crossed homomorphism.
Moreover, from \eqref{eq:C} we see that $W_+(\gamma, z, \lambda)$ is holomorphic in terms of 
$z \in \mathbb{D}$.

\begin{theorem} 
Let $M= \pi_1(M) \backslash \mathbb{D}$  be a non-compact Riemann surface. Then for any horizontally liftable minimal Lagrangian immersion $f: M\rightarrow \mathbb{C}P^2$, 
there exists some holomorphic function $U_+(\cdot, \lambda)$
: $\mathbb{D}\rightarrow \Lambda^{+} SL(3,\mathbb{C})_{\sigma}$ such that
\begin{equation}\label{eq:W+spliting}
 W_+(\gamma, z, \lambda)=U_+(z,\lambda) U_+(\gamma\cdot z, \lambda)^{-1}
\end{equation} 
holds.
\end{theorem}

\begin{proof}
Following the proof of Theorem 3.2 of \cite{DH03} or the proof of Theorem 31.2 of \cite{Forster} and using Theorem 8.2 in \cite{Bungart} which implies the vanishing of $H^1(\mathbb{D}, \Lambda^+SL(3,\mathbb{C})_{\sigma})$, one obtains that $W_+(\gamma,z,\lambda)$ splits in $\Lambda^+ SL(3,\mathbb{C})_{\sigma}$.
\end{proof}

\begin{corollary}
 By defining 
\begin{equation}
C(z,\lambda)= C_0(z,\lambda) U_+(z,\lambda),
\end{equation}
we get a holomorphic map $C(z,\lambda)$ satisfying
\begin{equation}
C(\gamma . z, \lambda)=\chi(\gamma, \lambda) C(z,\lambda)
\end{equation}
for any $\gamma\in \pi_1(M)$.
\end{corollary}

By abuse of language, a matrix function $C$ satisfying the last equation above will be  called an \emph{extended invariant holomorphic frame} of $f$.

As a consequence, we have 

\begin{corollary}
 $\eta=C^{-1} dC$ is an invariant holomorphic potential under $\pi_1(M)$, i.e.,
 $\gamma^*\eta=\eta$
 for any $\gamma\in \pi_1(M)$.
 \end{corollary}

Note, in the case under consideration we have  unitary monodromy and the closing condition 
$\chi(\gamma, \lambda=1)= c_\gamma I$ for $\gamma\in \pi_1(M)$ and some $c_\gamma  \in S^1$ satisfying $c_\gamma^3 =1$.
Summing up we have

\begin{theorem}\label{Thm 5.10}
Each minimal Lagrangian
immersion  from some non-compact Riemann surface $M$ into $\C P^2$
can be obtained by the loop group method from some 
invariant holomorphic potential defined on the universal cover $\mathbb{D}$ of $M$, i.e., some holomorphic potential on $\mathbb{D}$ satisfying $\gamma^*\eta=\eta$
 for any $\gamma\in \pi_1(M)$.
\end{theorem}

\begin{remark}
 The proof for the existence of an extended invariant frame in the case of non-compact $M$ given above is a bit round about. In some cases, like in the case of $k$-noids, one can actually realize 
$\mathcal{T}_k$ as an open subset of $\C$ and thus can define the coordinate frame directly  using the coordinates of $\C$. Then the  pull back to the universal cover will yield an invariant frame. In all cases, the transition to extended frames (i.e. the introduction of the loop parameter) will preserve this property (since ``$k$'' does not change in this transition).

 Another general proof for non-compact $M$ goes as follows: Consider $k$ as stated explicitly above, (\ref{eq:k}). Clearly, all entries of $k$ are quotients of $j(z) = cz + d$ and of its complex conjugate, where $j$ is defined by the denominator of the transformation $\gamma$.
It is straightforward to verify that $j$ is a multiplicative holomorphic crossed homomorphism (which does not depend on $\lambda$ and is $\C^*$-valued). Therefore by the multiplicative version of \cite{Forster}, Theorem 28.4 we obtain that $j$ is a boundary. From this it is easy to see that actually also $k$ is a boundary. 
\end{remark}

%%%%%%%%%%%%%%%%%%%%%%%%%%%
\subsection{Invariant frames and potentials in the case where $M$ is compact}
\label{sec:compact}

In this section we assume that $M$ is compact and also $M \neq S^2$ (as always in this paper). Then the universal cover of $M$ is $\mathbb{C}$ or $\mathbb{D}$, in particular, non-compact. Then by Theorem 2.6.8 in \cite{Miyake} we obtain a meromorphic function $h:\D \rightarrow \C^*$ such that the splitting  \eqref{eq:j-spliting} holds for $j(\gamma, z)$ and $\gamma\in \pi_1(M)$. Thus we have

\begin{theorem}
Assume that $M\neq S^2$ is compact. For any frame $\F$ satisfying (\ref{transF})
there exists  a meromorphic function $h:\D\rightarrow \C^*$ and a function $p$ on $\D$ such that 
$k(\gamma, z, \bar{z})=p(z, \bar{z})p(\gamma . z, \overline{\gamma . z})^{-1}$, where $p(z, \bar{z})=\mathrm{diag}(\frac{h(z)}{\overline{h(z)}}, \frac{\overline{h(z)}}{h(z)}, 1)$ and $j(\gamma, z)=h(z)h(\gamma\cdot z)^{-1}$. 
\end{theorem}

\begin{corollary} \label{cor5.10}
For compact $M\neq S^2$, using the notation above we set $\hat{\F}(z, \bar{z}) = \F(z, \bar{z}) p(z, \bar{z})$.
Then we obtain for all $\gamma \in \pi_1(M)$:
\begin{equation}
\gamma^* \hat{\F} = c_\gamma  \hat{\F}.
\end{equation}
\end{corollary}

Frames with the property just stated are called
\emph{invariant frames}  with regard to $\pi_1(M)$.

As in the non-compact case, we consider next  holomorphic extended frames.
Proceeding as in the case, where $M$ is non-compact, we also write in the compact case
\begin{equation*}
\F (z,\bar{z},\lambda) = C_0 (z,\lambda)V_+ (z,\bar{z},\lambda),
\end{equation*}
and obtain the transformation formula
\begin{equation}
C_0(\gamma \cdot z,\lambda)=
\chi(\gamma, \lambda) C_0(z,\lambda) W_+(\gamma, z,\lambda). \label{eq:CC}
\end{equation}

As in the non-compact case  we want to find some $U_+ \in 
\Lambda SL(3,\C)_{\sigma}$  such that \eqref{eq:W+spliting} holds.  It turns out that in the present case this is only possible by choosing some meromorphic  $U_+$. 
We thus want to prove

\begin{theorem}\label{thm:cpt}
If $M\neq S^2$ is a compact Riemann surface, then for any minimal Lagrangian immersion $f: M\rightarrow \C P^2$,  there exists a meromorphic potential for $f$ which is invariant under $\pi_1(M)$. 
\end{theorem}
\begin{proof}

The proof of the claim proceeds in several steps.

Step 1. Since $W_+$ is a cocycle, also the first term $W_+(z, \lambda = 0)$ 
of $W_+$ is a cocycle. This can be split by meromorphic functions. Hence, adjusting $C_0$ correspondingly we can assume without loss of generality that $W_+(z, \lambda =0) = I$.

Step 2. Let $\pi: \D\rightarrow M$ denote the universal covering of $M$. We know that $\pi$ is Galois and $\mathrm{Deck}(\D/M)$ is isomorphic to $\pi_1(M)$.
As in 28.3 of \cite{Forster}, one can choose an open covering $\mathcal{U}=(U_i)_{i\in I}$ of $M$ and $\pi_1(M)$-charts $\phi_i: \pi^{-1}(U_i)\rightarrow U_i\times \pi_1(M)$ with $\phi_i(z)=(\pi(z), \eta_i)$, where $\eta_i: \pi^{-1}(U_i)\rightarrow \pi_1(M)$ satisfies 
$$\eta_i(\gamma . z)=\gamma\eta_i(z), \text{ for any } z\in \pi^{-1}(U_i), \gamma\in \pi_1(M). $$
Define functions $\tilde{\Psi}_i: \pi^{-1}(U_i)\rightarrow \Lambda^{+}SL(3,\C)_{\sigma}$ by
$$\tilde{\Psi}_i(z, \lambda):=W_{+}(\eta_i(z)^{-1}, z, \lambda)^{-1}.$$
Now we claim that 
$$\tilde{\Psi}_i(\gamma . z,\lambda)=\tilde{\Psi}_i(z, \lambda) W_{+}(\gamma, z,\lambda)$$ 
for any $\gamma\in \pi_1(M)$. In fact, for any $z\in \pi^{-1}(U_i)$, we have by definition
$$\tilde{\Psi}_i(\gamma . z, \lambda)=W_{+}(\eta_i(\gamma . z)^{-1}, \gamma . z, \lambda)^{-1}
=W_{+}(\eta_i(z)^{-1}\gamma^{-1}, \gamma . z, \lambda)^{-1}.$$
The fact that $W_{+}(\gamma, z,\lambda)$ is a crossed homomorphism implies that
\begin{align*}
\tilde{\Psi}_i(\gamma . z, \lambda)&=(W_{+}(\gamma^{-1},\gamma . z, \lambda) W_{+}(\eta_i(z)^{-1},z, \lambda))^{-1}\\
&=W_{+}(\eta_i(z)^{-1},z, \lambda)^{-1}W_{+}(\gamma^{-1}, \gamma . z, \lambda)^{-1}\\
&=\tilde{\Psi}_i(z,\lambda) W_{+}(\gamma, z, \lambda),
\end{align*}
where the last equality follows from the fact that
$$I=W_{+}(\gamma,z, \lambda)W_{+}(\gamma^{-1}, \gamma . z, \lambda).$$

Step 3.  As a consequence we see that the functions
$$\tilde{g}_{ij} : \pi^{-1}(U_i)\cap \pi^{-1}(U_j)\rightarrow \Lambda^{+}SL(3,\mathcal{O}(\pi^{-1}(U_i)\cap \pi^{-1}(U_j)))$$
defined by 
$$\tilde{g}_{ij} :=\tilde{\Psi}_i\tilde{\Psi}_j^{-1}$$
are holomorphic in $z$ and invariant under the action of $\pi_1(M)$.

So $\{ \tilde{g}_{ij} \}$ descends to a cocycle   $\{g_{ij}\}$ on $M$ and the $g_{ij}$'s  are holomorphic in $z$ and $\lambda$  in the unit disk $\mathcal{D}$ and invertible for  $\lambda$.
The cocycle $g_{ij}$ defines a holomorphic vector bundle on $M$ which depends holomorphically on $\lambda\in \mathcal{D}$. The trivialization functions $g_{ij}(z,\lambda)\in \Lambda^+ SL(n,\mathbb{C})_{\sigma}$ and $g_{ij}(z, \lambda=0)=I$. 

By Proposition 3.12 of \cite{Rohrl}, this vector bundle is meromorphically (in $z$) trivial  and 
denpends holomorphically in $\lambda\in \mathcal{D}$. Thus there exists  $p_j\in \Lambda GL(3, \mathcal{M}(U_j))$ are meromorphic in $z$ and 
holomorphic  in $\lambda$ and invertible for $\lambda$ in the unit disk, such that
$$g_{ij}=p_i^{-1}p_j $$
defined on $U_i\cap U_j$.
Notice that $\det p_i|_{U_i\cap U_j}=\det p_j|_{U_i\cap U_j}$ and 
$p_i|_{U_i\cap U_j}(z,\lambda=0)=p_j|_{U_i\cap U_j}(z,\lambda=0)$ on $U_i\cap U_j$.
Therefore there is a global function $f: M\rightarrow \mathbb{C}$ and a constant matrix $p\in SL(3,\mathbb{C})$ on $M$ such that $f|_{U_j}=\det p_j$ and $p|_{U_j}=p_j(z,\lambda=0)$ for any $z\in U_j$. In fact, $f$ and $p$ are constant since $M$ is compact. 
Without loss of generality, we can assume $\det p_j=1$ and $p_j(z,\lambda=0)=I$ for all $z\in U_j$, i.e., 
$p_j \in \Lambda^+ SL(3,\mathbb{C})$ with the leading term $I$.

Step 4. All there is missing now is the property of  $p_j $ to be twisted in $\lambda$ relative to $\sigma$.
 
 Now replacing $\lambda$ by $t\lambda$ for $0<t \leq1$,  we get 
 $g_{ij}$ and $p_j$ such that 
  \begin{equation}\label{eq:p-twister}
 p_j(z,t\lambda)=p_i(z, t\lambda) g_{ij}(z, t\lambda).
 \end{equation}
  For $t>0$ small enough we can write
 $$p_j(z, t\lambda)=\exp X_j^{\rm{fix}}(z, t\lambda)\exp X^{\rm{rest}}_j(z, t\lambda),$$
 where $X_j^{\rm{fix}}(z, t\lambda)$ denotes the element in the loop Lie algebra fixed by $\sigma$ and $X^{\rm{rest}}_j(z, t\lambda)$ denotes the element in the complementary space. 
Then since the fixed set of $\sigma$ is a subalgebra, \eqref{eq:p-twister} implies 
$$p_j^{\rm{fix}}(z, t\lambda)=p_i^{\rm{fix}}(z, t\lambda) g_{ij}(z,t\lambda), \quad p_j^{\rm{rest}}(z, t\lambda)=p_i^{\rm{rest}}(z,t\lambda)$$
on $U_i\cap U_j$.
This yields a global matrix $p^{\rm{rest}}$ definded on $M$ such that 
$p^{\rm{rest}}|_{U_j}=p_j^{\rm{rest}}$. 
Then by the compactness of $M$, $p^{\rm{rest}}$ is constant. 
Now replace $p_j$ by $p_j (p^{\rm{rest}})^{-1}$.
Therefore we can assume without loss of generality that $p_j$ is fixed by $\sigma$ for all $\lambda$.
This finishes the proof.
\end{proof}

%%%%%%%%%%%%%%%%%%%%%%%%%%%%
\subsection{Riemann surfaces with abelian fundamental groups}
\label{sec:abelian pi_1}
The Riemann surfaces with abelian fundamental groups are well known:

1.The simply connected Riemann surfaces $S^2$, $\C,$ and  $\mathcal{D} = \{w\in\mathbb{C}| |w| < 1\}$.

2. a) The doubly infinite cylinder $\C^*$ with universal covering map $\C \rightarrow \C^*$, 
$w \rightarrow e^{iw}$, and fundamental group $\mathbb{Z}$,  acting on $\C$ by $w \rightarrow w + 2 \pi$.

     b) The onesided-finite cylinder $\mathcal{D}^*$ with universal covering map 
     $\mathcal{H} = \{w\in\mathbb{C}| \mathrm{Im} w >0 \} \rightarrow \mathcal{D}^*, w \rightarrow e^{iw}$, and fundamental group
     $\mathbb{Z}$ acting by  $w \rightarrow w + 2 \pi.$
 
     c) The twosided-finite cylinders $\mathbb{A}_r = \{z \in \C| r < |z| < 1\}, 0 < r <1,$ with universal covering map $\mathcal{H} 
     \rightarrow \mathbb{A}_r,
     w 
     \mapsto \exp(2 \pi i  \frac{\log w}{log \rho} )$, $\rho=|z|$, and fundamental group 
     $\mathbb{Z}$ acting by $w \rightarrow \rho w$. Note here $r = \exp(\frac{-2\pi^2}{\log \rho})$.
    
 3. The tori $\mathbb{T}_\tau, \mathrm{Im} \tau > 0, $ with universal covering map
  $\C \rightarrow \C / \mathcal{L}_\tau,$ fundemamntal group 
  $ \mathcal{L}_\tau = \mathbb{Z} \oplus \tau \mathbb{Z}$ and group action generated by
   $w \rightarrow w + 1$ and $w \rightarrow w + \tau$.
   
For more details we refer to \cite{Farkas-Kra}, section IV.6.
   
We would like to point out that the loop group method respects the conformal structure and thus produces all harmonic maps (respectively, minimal Lagrangian surfaces) from Riemann surfaces,
while preserving the conformal structure of the surface.

%%%%%%%%%%%%%%%%%%%%%%%%%%%%%%%
\section{Perturbed equivariant minimal Lagrangian cylinders in $\C P^2$}
\label{sec:pertubed}

In this section, we discuss the construction of perturbed equivariant  minimal Lagrangian surfaces. We first recall the characterization of equivariant minimal Lagrangian surfaces in $\mathbb{C}P^2$ and then single out the types of equivariant surfaces we want to perturb. 
%%%%%%%%%%%%%%%%%%%%%%%%%%%%%%
\subsection{The basic examples of equivariant minimal Lagrangian immersions in \texorpdfstring{$\C P^2$}{}}
\label{subsec:6.1}

In \cite{DoMaNL} and \cite{DoMaEX} the authors have discussed equivariant 
minimal Lagrangian immersions from the point of view of the loop group method and with
applications to equivariant cylinders and tori.

Referring to Section 4 of \cite{DoMaEX}, we start by defining an equivariant immersion
 as a mapping $f: M \rightarrow \C P^2$ that satisfies the condition  $f(\gamma_t.p)= R_t f(p)$ for all $p\in M$ and $t\in \mathbb{R}$, with 1-parameter group $(\gamma_t, R_t) \in (\mathrm{Aut}(M)\times \mathrm{Iso}(\C P^2))$.
A Riemann surface $M$  endowed with such an equivariant immersion $f$ is called an equivariant surface.

Classical complex analysis ( refer to \cite{Farkas-Kra}, Section IV.6.) then yields a complete list of surfaces $M$ that admit  1-parameter groups $\gamma_t\in \mathrm{Aut}(M)$.
 It turns out that there are two types of equivariant surfaces:
 
 (T)  Translationally equivariant surfaces.
 
  These are  those equivariant surfaces which can be represented on  a strip $\St$ and admit the operation
  $z \rightarrow  z + t$ as a group of symmetries,
  
  and 
  
  (R)  Rotationally equivariant surfaces.
  
  These are  those equivariant surfaces which admit  a 1-parameter group of rotations.
  
  There are three types of rotationally equivariant surfaces:
  
  (R0) These are the surfaces $M$, for which the group of rotations has a fixed point not contained in $M$, 
  like the point $z = 0$ for $ M =\C^*$, or $\D^*$ or an annulus.
  
  (R1)  These are the surfaces $M$, for which the group of rotations has exactly one fixed point  contained in $M$, like $z = 0$ for $M = \C$ or  $\D$.
  
  (R2) These are the surfaces $M$, for which the group of rotations has two fixed points  contained in $M.$
  
  Actually, only $M = S^2$, satisfies $(R2)$.

Summing up the results of section 4.3  and the last paragraph of section 4.2 of \cite{DoMaEX} we obtain: 
\begin{theorem}
Any  minimal Lagrangian immersion $f$ from $\C$ or $\D$ or $S^2$ into $\C P^2$ which is rotationally equivariant has a vanishing cubic Hopf differential, and therefore  is  totally  geodesic in $\C P^2$ and  its image is, up to isometries of $\C P^2$, contained in  $\R P^2$.
\end{theorem}

Actually, the case $M = \D$ extends to the case $M = \C$ by Section 4.2 of \cite{DoMaEX}.

{\bf In this paper we only consider translationally equivariant surfaces and rotationally equivariant surfaces
of type $(R0)$.}
\vspace{2mm}

%%%%%%%%%%%%%%%%%%%%%%%%%%%%
\subsection{Periodic translationally equivariant minimal Lagrangian surfaces and rotationally equivariant surfaces of type $(R0)$}

Using the covering map
$$ \tilde{\pi}: \C \rightarrow \C^*, w \rightarrow \exp(iw),$$ 
we have the following relation between these two types of equivariant surfaces, which is essentially Theorem 4 of \cite{DoMaEX}.

\begin{theorem}
Any rotationally equivariant minimal Lagrangian immersion $f:M \rightarrow \C P^2$ defined 
on $M =\C^*, \mathbf{D}^*, \mathbf{D}_r$
can be extended without loss of generality to $\C^*$ and can be obtained from some 
$ 2 \pi$-periodic translationally equivariant minimal Lagrangian immersion defined on $\C$ by projection.
\end{theorem}

Translationally equivariant minimal Lagrangian surfaces can be constructed by the loop group method. We have shown

\begin{proposition} [Proposition 4 of \cite{DoMaNL}]
Up to isometries of $\mathbb{C}P^2$,  any translationally equivariant minimal Lagrangian surface 
can be generated by a potential $\tilde{\eta} (w, \lambda) = D(\lambda) dw,$ defined on $\C,$ and  with 
$D(\lambda) \in su(3)$ of the form 
\begin{equation}\label{eq:Delaunay}
D(\lambda)=\begin{pmatrix}0&-\lambda \bar{b}& \lambda^{-1}a\\
\lambda^{-1}b& 0& -\lambda \bar{a}\\
-\lambda \bar{a}& \lambda^{-1}a & 0
\end{pmatrix},
\end{equation}
where $a$ is a non-zero purely imaginary constant such that $-ia>0$ and   $b=\frac{i\psi}{a^2}$ is  
constant. 

Moreover, for the characteristic polynomial of $D(\lambda)$ we obtain
\begin{equation} \label{charpolD}
\det(\mu I-D(\lambda))=\mu^3+\beta \mu-2i \mathrm{Re}(\lambda^{-3}\psi),
\end{equation}
with $\beta=2|a|^2+|b|^2$.
\end{proposition}

\begin{remark} \label{a and b}
$(1)$ The diagonal of $D(\lambda)$ vanishes, since we have chosen the base point $0 \in \C$ such that the metric of the corresponding surface has a
critical point at $w = 0$ (see Section 5.3 of \cite{DoMaNL}).

$(2)$ We always assume $b \neq 0$, which is equivalent to  $\psi \neq 0$, 
as  otherwise the image of $f$ would constitute an open portion of the real projective plane (see, for instance, Remark 4 of \cite{DoMaNL}).

$(3)$ The surface associated with the potential $\tilde{\eta}$ stated above is an immersion, since $a \neq 0$.

$(4)$ All translationally equivariant surfaces, considered on the maximal domain of 
definition $\mathbb{C}$,
are full and have a global horizontal lift  (see Proposition 3 in \cite{DoMaNL} and Theorem \ref{non-compact mLi Legendrian}).  
\end{remark}

We end this subsection by presenting several geometric results.

\begin{proposition}
The translationally equivariant immersion $\tilde{f}: \C \rightarrow \C P^2$ defined from 
$\tilde{\eta}(w,\lambda) = D(\lambda) dw$ above is 
full   if and only if $a\neq 0$ and $b\neq 0$.
\end{proposition}

\begin{proof}  
We want to apply the definition of the notion {\bf full} stated in Section 4.1 and choose some 
$\mathcal{R} \in Iso_0(\C P^2)$ satisfying $\mathcal{R} \tilde{f}(w) = \tilde{f}(w)$ for all $w \in \C$.
For the global horizontal lift $\tilde{\f}$ of $\tilde{f}$ this implies 
$R \tilde{\f}(w) = \alpha(R)\tilde{\f}(w)$ for some   
$ \alpha(R)  \in \C^*, |  \alpha(R) | = 1, $ and $R \in SU(3),$ where $\alpha$ does not depend on $w$, 
 since the relation between $\mathcal{R}$ and $R$ does not depend on $f$ and thus not on $w$.
  Next we use the definition of the frames $\mathcal{F}$ and $\F$ given in Section 2.5 of  \cite{DoMaB1}. 
  Since $\alpha$ does not depend on $w$, for the frame $\F$ satisfying $\F(0,0) = I$ this implies
 $R \F(w, \bar w) = \alpha(R)\F(w, \bar w) $.  An evaluation at $w = 0$ implies $R  = \alpha(R) I,$ whence 
 $\alpha (R) \in U(1)$.
 Since the isometries are in $PSU(3)$, $\mathcal{R} =\mathrm{id}$ follows. 
\end{proof}

From Lemma 1 of \cite{DoMaNL},  we derive:

\begin{theorem}
The translationally equivariant immersion $\tilde{f}: \C \rightarrow \C P^2$ defined from 
$\tilde{\eta}(w,\lambda) = D(\lambda) dw$ above is non-flat if and only if the characteristic 
polynomial  of $D(\lambda)$ has three different roots for all $\lambda \in S^1$.
Furthermore, 
the root $0$ occurs if and only if 
$\lambda^{-3} \psi $ is purely imaginary. This can only happen for 
 the six different values of $\lambda \in S^1$  satisfying $\lambda^6 = - \frac{\psi}{\bar \psi}$.
\end{theorem}

Note that any flat minimal Lagrangian surface in $\mathbb{C}P^2$ is a parametrization of an open subset of the Clifford torus, up to isometries of $\mathbb{C}P^2$. 
For the Clifford torus, we have $a=b=i$ and $\psi=-1$, thus 
$D(\lambda)=\begin{pmatrix}
0&i\lambda & i\lambda^{-1}\\
i\lambda^{-1}& 0& i\lambda \\
i \lambda & i\lambda^{-1} & 0
\end{pmatrix}$. 
Moreover, the eigenvalues of $D(\lambda=1)$ are $i, i, -2i$.

{\bf In the rest of this section we only consider translationally equivariant 
immersions which are not flat and satisfy $\psi \neq 0$.}

Finally we consider periodic translational equivariant minimal Lagrangian immersions.

\begin{remark}
Since we consider here only one equivariant minimal Lagrangian immersion and (just below) only 
one perturbed equivariant minimal Lagrangian immersions (as opposed to more complex situations, like minimal Lagrangian trinoids etc.) the description of the case of periodic equivariant  minimal Lagrangian immersions is easy to describe.

Let $id_1(\lambda), id_2(\lambda)$ and $id_3(\lambda)$ denote the eigenvalues of $D(\lambda)$, then the results of \cite{DoMaNL} and \cite{DoMaEX} imply that the corresponding immersion $\tilde{f}$ descends
to  the cylinder $\C / p \mathbb{Z}$ , making the following diagram commutative
 \[
  \begin{diagram}
    \node{\C}   \arrow{s,l}{\hat{\pi}} \arrow{se,l}{\tilde{f}} \\
    \node{\C / p \mathbb{Z}}  \arrow{r,l}{f}\node{\C P^2} 
   \end{diagram}
\]
 if and only if $\tilde{f}$ is invariant under the translation by  $p >0$ if and only if  $p>0$ satisfies 
 $ p d_j(\lambda = 1) = 2 \pi k_j$, $j = 1, 2, 3$ for some integers $k_j$.
 \end{remark}

We recall from Corollary 2 of \cite{DoMaNL}

\begin{theorem}
Retaining the notation used  above, such a $p>0$ exists if and only if 
 $d_1(\lambda =1) / d_2(\lambda =1)$ is rational.
 More precisely, writing $d_1(\lambda =1) / d_2(\lambda =1) = l_1/l_2 $ with relatively prime
  integers $l_1, l_2$ we obtain 
   $\frac{p d_1(\lambda = 1) }{2 \pi l_1} = \frac{p d_2(\lambda = 1) }{2 \pi l_2}$.
As a consequence, the period is, as stated above, of the form 
 $ p d_j(\lambda = 1) = 2 \pi l_j$, $j = 1, 2.$
\end{theorem}

As a corollary we obtain

\begin{theorem}
After some scaling by a positive real number, each periodic translationally  
equivariant minimal Lagrangian immersion is the universal cover of some 
rotationally equivariant minimal Lagrangian immersion of type $(R0)$.
\end{theorem}

\begin{corollary}
The rotationally equivariant immersions of type $(R0)$ correspond in a 1-1-relation to the translationally equivariant minimal Lagrangian $2 \pi$-periodic immersions.
\end{corollary}

\begin{remark}
$(a)$ To explain the relation between the rotationally equivariant surface and the translationally equivariant surface occurring in the corollary above, one should use the  universal cover map
$$\tilde{\pi} : \C \rightarrow \C^*,  w\mapsto z=e^{iw}$$
to relate the corresponding natural potentials
$\eta $ on $\C^*$ and $\tilde{\eta}$ on $\C$.
Let  
$$\tilde{\eta}(w, \lambda) =  D(\lambda) dw$$ 
denote the holomorphic potential on $\C$ for the 
 associated  translationally equivariant minimal Lagrangian immersion defined on $\C$. Then
$$ \tilde{\pi}^*(\eta) = \tilde{\eta}.$$
is equivalent  to  
$$\tilde{\eta}(w, \lambda) =  \frac{-i D(\lambda)}{z} dz.$$ 

$(b)$ More precisely, one can pull back any rotationally equivariant immersion
 $f: \C^* \rightarrow \C P^2$ 
of type $(R0)$ to a translationally equivariant immersion $ \tilde{f}: \C \rightarrow \C P^2$ 
by putting $\tilde{f}(w) = f(\tilde{\pi }(w))$. 
Note that $\tilde{f}$ is invariant under translation by $2 \pi$.
Conversely, given a translationally equivariant immersion $ \tilde{f}: \C \rightarrow \C P^2$ 
which is invariant under translations by $2 \pi$ one can push down this immersion to a
rotationally  equivariant immersion  $f: \C^* \rightarrow \C P^2$  defined on $\C^*$.

$(c)$ All the immersions  $\tilde{f}$ considered above can be constructed by the loop group method 
as outlined in Section \ref{subsec:construction}.
\end{remark}

%%%%%%%%%%%%%%%%%%%%%%%%%%%%
\subsection{Perturbed equivariant minimal Lagrangian surfaces in \texorpdfstring{$\C P^2$}{}}

In this subsection, we discuss minimal Lagrangian surfaces which are constructed 
from perturbed equivariant potentials.
Our approach generalizes arguments given in \cite{DW08}.
When comparing our work to the work done in loc.cit. one needs to keep in mind that 
there Hermitian Delaunay matrices were used, while here we use skew-hermitian matrices.

By the explanation of the next to  last subsection, we will perturb the  rotationally 
equivariant immersions of type $(R0)$.
Also, in order to keep the notation and the arguments as simple as possible, 
we use here somewhat special Delaunay matrices (see below).
The most general case can be dealt with by using the r-loop group formalism which has not 
yet been introduced into the discussion of minimal Lagrangian surfaces in $\C P^2$.

\begin{definition}
Let $\mathcal{D}_R$ be  a disk in $\mathbb{C}$ about $z=0$ with radius $R>0$.
Let $\eta(z,\lambda)$ be a holomorphic potential for a minimal Lagrangian surface 
in $\C P^2$ which is defined on $\mathcal{D}_R\backslash \{0\}$ and therefore has the form
\begin{equation}\label{perturbed potential}
\eta(z,\lambda) = \frac{-iD(\lambda)}{z} dz + \eta_+(z,\lambda), 
\end{equation}
where
 \begin{equation*}
\eta_+ (z, \lambda) =\sum_{m=0}^\infty \eta_{+,m}(\lambda) z^m dz
\end{equation*}
is holomorphic on $\mathcal{D}_R$ in $z$ and holomorphic in $\lambda\in \mathbb{C}^*$ with a pole in $\lambda$ at $\lambda = 0$ of order at most $1$. Moreover, 
 $D(\lambda)$ is a Delaunay matrix with $a,b\neq 0$. Then, in particular, the corresponding 
minimal Lagrangian immersion is full.
Assume, moreover, that  the $(1,3)$-coefficient of the $\lambda^{-1}$-term in 
$\eta$ never vanishes for all $z\in \mathcal{D}_R\backslash\{0\}$, which leads to an immersed surface. Such a potential is called a \emph{perturbed Delaunay potential} defined on $\mathcal{D}_R \backslash \{0\}$.
\end{definition}

For the record we state 
\begin{theorem}
The potential \eqref{perturbed potential} produces by the loop group method (Theorem \ref{Th:DPW-hol}) a full minimal Lagrangian immersion $f_{\eta}$ on $\mathcal{D}_R \backslash \{ x\in \mathcal{D}_R | x\leq 0\}$ to $\mathbb{C}P^2$. This minimal Lagrangian immersion is globally horizontally liftable.
\end{theorem}

Finally we address the natural lift of $f_{\eta}$ to the (extended) universal cover
\begin{equation}
\pi: \mathcal{S}_{\eta}=\{w\in\mathbb{C}| \mathrm{Re}( w) \in \mathbb{R}, -\ln R < \mathrm{Im}(w) <\infty\}\rightarrow \mathcal{D}_R\backslash\{0\},\quad w\mapsto e^{iw}.
\end{equation}
The pull back of $\eta$ by $\pi$ will be denoted by $\tilde{\eta}$.
Since we have assumed that $\eta$ is defined on $\mathcal{D}_R\backslash\{0\}$, we obtain for all 
$z\in \mathcal{S}_{\eta}$, $\lambda\in\mathbb{C}^*$ and all  $n \in \mathbb{Z}$ 
\begin{equation}
\tilde{\eta}(z+2\pi n,\lambda)=\tilde{\eta}(z,\lambda).
\end{equation}

As a consequence we infer
\begin{lemma}
Let $\tilde{f}_{\eta}$ denote the (extended) lift of $f_{\eta}$ to $\mathcal{S}_{\eta}$, then there exists 
$\chi (\lambda) \in  \Lambda SL(3,\C)_{\sigma}$:
$$\tilde{f}_{\eta}(z+2\pi n, \lambda)=\chi(\lambda)^{n}  \tilde{f}_{\eta}(z,\lambda),$$
for all $n\in \Z$.
\end{lemma}

For the $*$-perturbed potentials defined below we will show that $\chi(\lambda)$ 
 is unitary and can be
expressed using the Delaunay matrix $D(\lambda)$.

%%%%%%%%%%%%%%%%%%%%%%%
\subsection{$*$-perturbed  minimal Lagrangian immersions into \texorpdfstring{$\C P^2$}{23}}

We have discussed equivariant minimal Lagrangian immersions in detail in  \cite{DoMaNL} and \cite{DoMaEX}. For the examples illustrating this paper we will restrict to those immersions which are easiest to handle technically. 

\begin{definition} \label{def:*-Delaunay} 
Let $D(\lambda)$ be a Delaunay matrix. Recall from Remark \ref{a and b}
that we assume that the parameters $a$ and $b$ defining $D(\lambda)$ are both different from zero.
Moreover, if the eigenvalues of $-iD(\lambda)$ are all integers for $\lambda=1$,
then we say that $D(\lambda)$ is a \emph{$*$-Delaunay matrix}.
 
 Let $\eta(z,\lambda) = \frac{-iD(\lambda)}{z}  dz + \eta_+(z,\lambda) $ be a perturbed Delaunay  potential defined on $\mathcal{D}_R\backslash \{0\}$.
Then $\eta$ is said to be  a \emph{$*$-perturbed Delaunay potential} if 
\begin{itemize}
\item[(i)] $D(\lambda)$ is a $*$-Delaunay matrix.
\item[(ii)] There exists a positive integer  $N$ such that $j\mathrm{Id}+\mathrm{ad}(-iD(\lambda))$ is invertible for any $j>N$ and $\lambda \in S^1$. 
\item[(iii)]  $\eta_+(z,\lambda)= \mathring{\eta}_+ (z,\lambda)dz$, where $\mathring{\eta}_+(z,\lambda)=
\sum_{k=N}^{\infty} \mathring{\eta}_{+,k}(\lambda)  z^k \in \Lambda sl(3,\mathbb{C})_{\sigma} $
 with $k$ starting from the same $N$ as $j$  in (ii). 
\end{itemize}
The corresponding minimal Lagrangian immersion will be called a $*$-perturbed equivariant minimal Lagrangian immersion.
\end{definition}

\begin{remark}
\begin{enumerate}
\item For any given $*$-Delaunay matrix $D(\lambda)$,  the desired integer $N$ in (ii)  always exists by 
the compactness of $S^1$.
\item It is easy to see that in  (ii) it suffices to require  the condition on $\lambda$  
for $\lambda \in S^1$ only,
since then the condition holds for a (possibly small open) annulus containing $S^1$.
\item We would like to point out that in the definition above we
allow that the diagonal of $D(\lambda)$ 
does not vanish. This will be of importance if one considers several Delaunay matrices simultaneously.
When dealing with a single Delaunay matrix we will always assume that its diagonal vanishes.
\end{enumerate}
\end{remark}

Now we consider a $*$-perturbed Delaunay potential $\eta$ and present a special solution to the differential equation 
$dC=C\eta$. Namely, we prove

\begin{theorem}\label{thm:Sch-Sch}
Let $\eta=\frac{-i D(\lambda)}{z} dz+\mathring{\eta}_{+}(z,\lambda)dz$ be a $*$-perturbed Delaunay potential 
defined on   $\mathcal{D}_R\backslash\{0\}$,  where $\mathring{\eta}_+(z,\lambda)=
\sum_{k=N}^{\infty} \mathring{\eta}_{+,k}(\lambda)  z^k \in \Lambda sl(3,\mathbb{C})_{\sigma} $.
 Then  for any straight line segment $L_0 \subset \mathcal{D}_R$ connecting the origin to 
 a point on the circle of radius $R$
 there exists on $ \mathcal{D}_R\backslash L_0$ 
 a solution $H(z,\lambda)\in \Lambda SL(3,\mathbb{C})_{\sigma}$ to the differential equation  
\begin{equation}\label{eq:H}
dH = H \eta, 
\end{equation} 
 of the form
\begin{equation}\label{eq:H solution}
 H(z, \lambda) = e^{\ln z (-iD(\lambda))} P(z,\lambda),
 \end{equation}
with a unique function $P(z,\lambda)$ holomorphic in $z\in \mathcal{D}_R$ 
satisfying $P(0,\lambda) = I$ 
for all $\lambda\in \mathbb{C}^*$ and having an expansion of the form
$$P(z, \lambda) = I + \sum_{k=N +1}^{\infty} P_k (\lambda)  z^k \in 
\Lambda SL(3,\mathbb{C})_{\sigma} . $$

Furthermore, since $D(\lambda)$ and $\mathring{\eta}_{+}(z,\lambda)$ are contained in the twisted loop algebra $\Lambda sl(3,\mathbb{C})_{\sigma}$, we also have
\begin{equation}
P(z, \lambda) \in \Lambda SL(3,\mathbb{C})_{\sigma} \quad \mbox{and}  \quad H(z,\lambda) \in \Lambda SL(3,\mathbb{C})_{\sigma}
\end{equation}
 for all  $\lambda \in \mathbb{C}^*$ and all $z \in \mathcal{D}_R$  for which the respective function is defined.
 \end{theorem}
 
 \begin{proof}  For the proof, we consider an increasing sequence of reduced domains of definition, namely   
  $\mathcal{D}_{R(m)} \backslash\{0\}$ with $m \in \mathbb{N}$ and $0 < R(m) < R(m+1) < R$. 
  The union of these domains covers 
 $\mathcal{D}_R\backslash\{0\}$. 
 For each domain, we can apply the proof outlined below, assuming a uniform norm for matrix functions across its boundary.
 It is easy to verify that from this argument we obtain the claim as stated above.
 
 To simplify notation we will drop the label  
``$R(m)$"  from here on. 
 We thus consider  the domain  $\mathfrak{D} \backslash\{0\} = \mathcal{D}_{R(m)}  \backslash\{0\}$. 
 
 1) \, We 
 begin by considering  the differential equation $dH=H\eta$  on $\mathfrak{D}$ and construct, for any given $L_0$ as in the claim, a solution of the form given by
 \eqref{eq:H solution}. 
This  is equivalent to solving, with $P$ holomorphic in $z \in  \mathfrak{D} = \mathcal{D}_{R(m)}$,
$$dP=P\eta-\frac{-iD(\lambda)}{z}Pdz,$$
which leads to the equation
\begin{equation}\label{eq:dP}
\frac{dP}{dz}=\frac{1}{z}[P,-iD]+P \mathring{\eta}_+.
\end{equation}

Consider the Banach space 
$\mathfrak{A} = \mathfrak{L} (\mathfrak{R})=\Lambda gl(3,\mathbb{C}),$
which denotes the set of all continuous endomorphisms of  $\mathfrak{R}=\Lambda{\C}^3$  endowed with a standard matrix Banach norm.
 Let  $ \mathfrak{B} = \mathfrak{L}(\mathfrak{A})$  be the matrix Banach algebra (with norm) 
  analogously defined  as in the case of $\mathfrak{L} (\mathfrak{R})$.
 
Introduce the commutator operator $\hat{D}$  defined by
$\hat{D}(\lambda)\xi:=[\xi,-iD(\lambda)] = \mathrm{ad} (i D(\lambda)) (\xi)$ and the operator
$\hat{Q}(z,\lambda)\xi:= \xi \mathring{\eta}_+(z,\lambda)$ for any $\xi\in \mathfrak{A}$. 
Thus the equation \eqref{eq:dP} can be written as
\begin{equation}\label{eq:P}
\frac{dP(z,\lambda)}{dz}=(\frac{1}{z}\hat{D}(\lambda)+\hat{Q}(z,\lambda))P(z,\lambda).\quad 
\end{equation}

Adjusting the argument of  Satz 1 of \cite{Sch-Sch} to our special assumptions,  
we obtain that \eqref{eq:H} has a 
solution in $\mathfrak{A}$, the untwisted Banach loop algebra $\Lambda gl(3,\mathbb{C})$.
 We present the adjusted argument for the convenience of the reader in detail.

First, according  to Hilfssatz a of  \cite{Sch-Sch}  and the  fact that we consider here actually  
any reduced domain $\mathfrak{D} \backslash\{0\}   = \mathcal{D}_{R(m)} \backslash\{0\}$, 
 we choose $n_\mathfrak{D} \in \mathbb{N}$ with  $n_\mathfrak{D} >N$ such that
for all $z\in \bar{\mathfrak{D}} \backslash\{0\} $ and all $\lambda \in S^1$ we have:
\begin{equation}\label{eq:n_D}
\frac{1}{n_{\mathfrak{D}} +1} ||\hat{D}||+\frac{R}{n_\mathfrak{D}+2}||\hat{Q}||<1.
\end{equation}
We can clearly choose the positive integers $n_\mathfrak{D}$ as an increasing sequence of integers (according to the suppressed label $m$). 

By our choice of $n_\mathfrak{D}$ we know that $n_\mathfrak{D} \mathrm{Id}-\hat{D}$ is invertible. 
Then by using the proof of Hilfssatz b of  \cite{Sch-Sch}, for any integer  $k \geq 0$, 
we can find a polynomial $p_{k}(z,\lambda)$  in $z$ with coefficients in $\mathfrak{A}$ 
and $\mathrm{deg}\, p_{k}\leq k$, $p_{k}(0,\lambda) = I$ such that
\begin{equation}\label{eq:p_k}
\frac{d}{dz}p_{k}(z,\lambda)= T(z,\lambda) p_{k}(z,\lambda)-z^{k} g_{k+1}(z,\lambda),
\end{equation}
where $T(z,\lambda):=\frac{1}{z}\hat{D}(\lambda)+\hat{Q}(z,\lambda)$ and $g_{k+1}$ is 
holomorphic in $z$ in an open subdisk $\mathfrak{D}'$ of $\mathcal{D}_R$ containing 
$\mathfrak{D}$ and takes  values in  $\mathfrak{A}$.

In fact, we construct by induction for $ k = 0,1,...$ a sequence $g_k(z, \lambda)$ defined on $\mathfrak{D}'$,
$q_k(\lambda)$ for $\lambda \in S^1$ and $p_k(z, \lambda)$ defined for the same arguments as $g_k$,
such that \eqref{eq:p_k} and the following three equations hold:
\begin{eqnarray}
(k\mathrm{Id}-\hat{D}(\lambda))q_k (\lambda)&=&g_k(0,\lambda), \label{eq:B'}\\
p_{k+1}(z,\lambda)&=&p_{k}(z,\lambda)+z^{k+1} q_{k+1}(\lambda), \label{eq:C'}\\
g_{k+1}(z,\lambda)&=&\frac{1}{z}(g_k(z,\lambda)-g_k(0,\lambda))+\hat{Q}(z,\lambda) q_k(\lambda). \label{eq:D'}
\end{eqnarray}
Recall that $k \mathrm{Id}-\hat{D}$ is singular  for  $k \leq N$ and nonsingular for $k > N$. 
 
More precisely, we start the induction for $k = 0$ by setting  $g_0(z,\lambda) = 0$, 
 $q_0 (\lambda) = I$ and
 $p_0(z, \lambda) = I$. 
 Then for $k= 1$
  we define  $g_1(z, \lambda)$ from \eqref{eq:D'}, obtaining  
  $g_1(z, \lambda)  = \mathring{\eta}_+(z, \lambda)$.
 Note that this is compatible with \eqref{eq:p_k}.  Moreover, 
from \eqref{eq:B'} we  choose 
$q_1(\lambda)=0$  which is compatible with    $g_1(0,\lambda)=0.$ 
Finally, from \eqref{eq:C'} we infer now
$p_1(z,\lambda)=I.$

We continue this procedure by choosing, starting from the level $k$, 
first $g_{k+1}(z, \lambda)$ from  \eqref{eq:D'}, then $q_{k+1}(\lambda)$ from  \eqref{eq:B'}, and finally $p_{k+1}(z, \lambda) $  from \eqref{eq:C'}.

For the case $1 <  k  \leq N$,
considering equations \eqref{eq:B'}  and \eqref{eq:D'}, 
we can (and will) choose $q_k(\lambda) = 0$ and then will obtain  $g_2(0, \lambda) = \cdots=g_N(0,\lambda)=0$ and $g_{N+1}(0,\lambda)=\mathring{\eta}_{+,N}(\lambda)\neq 0$,
since 
\begin{equation}
g_{k+1}(z,\lambda)=\frac{g_k(z,\lambda)}{z}
=\cdots=\frac{g_1(z,\lambda)}{z^k}=\frac{\mathring{\eta}_+(z,\lambda)}{z^k}
=\sum_{l=N}^{\infty}\mathring{\eta}_{+,l}(\lambda)z^{l-k}.
\end{equation}
Notice that $p_2(z,\lambda)=p_1(z,\lambda)+z^2 q_2(\lambda)=p_1(z,\lambda)=I$ and similarly, $p_N(z,\lambda)=I$. Hence 
it is clear that  $p_k(z, \lambda) = I$ for $1<k\leq N$.

For the range $k>N$, note that 
$k \mathrm{Id}-\hat{D}$ is  nonsingular.   
Thus we determine  in the  second  step 
of our induction procedure 
$q_{k+1}(\lambda)$ uniquely from $g_{k+1}(0,\lambda)$.
More precisely,  for $k > N$ we infer from  \eqref{eq:B'}, \eqref{eq:C'} and \eqref{eq:D'}
 \begin{equation*}
 \begin{aligned}
 q_{k}(\lambda)&=(k\mathrm{Id}-\hat{D} (\lambda))^{-1}\mathring{\eta}_{+, k-1}(\lambda), \\
 p_k(z,\lambda)&=I+\sum_{s=N+1}^{k} z^s (s \mathrm{Id}-\hat{D}(\lambda))^{-1} \mathring{\eta}_{+,s-1}(\lambda),\\
g_{k+1}(0,\lambda) &=\mathring{\eta}_{+,k}(\lambda).  
 \end{aligned}
 \end{equation*}
  This proves that the assertion \eqref{eq:p_k} holds for $k$ and 
completes the induction. 

Remark that for $n>n_\mathfrak{D} >N$,
$$p_n(z,\lambda)=I+z^{N+1} q_{N+1}(\lambda)+\cdots +z^n q_{n}(\lambda).$$

We now set $g(z,\lambda)=g_{n+1}(z,\lambda)$ for some $n>n_\mathfrak{D}> N$  
as above.

It follows from Hilfssatz a in  \cite{Sch-Sch} that there exists a unique holomorphic map $h(z,\lambda)$  with respect to $z\in \mathcal{D}_R$ such that $u(z,\lambda):=z^{n+1} h(z,\lambda)$ satisfies
$$\frac{d}{dz}u(z,\lambda)=T(z,\lambda)u(z,\lambda)+z^n g(z,\lambda).$$

Finally, 
$$P(z,\lambda)=p_{n}(z,\lambda)+u(z,\lambda)$$ 
provides a solution to \eqref{eq:P}.
 
Such a $P$ leads to a solution of \eqref{eq:H solution}. 
Moreover, $P$ has the form 
\begin{equation*}
\begin{aligned}
P(z,\lambda)&=I+\sum_{s=N+1}^n z^{s} q_s(\lambda) +z^{n+1}h(z,\lambda)\\
&=I+z^{N+1}((N+1)\mathrm{Id}-\hat{D}(\lambda))^{-1}\mathring{\eta}_{+,N}(\lambda)+\cdots\\
&+\cdots +z^n(n \mathrm{Id}-\hat{D}(\lambda))^{-1}\mathring{\eta}_{+,n-1}(\lambda)+z^{n+1}h(z,\lambda),
\end{aligned}
\end{equation*}
which implies that $P$ is uniquely defined.

\medskip

2) \, Recall that the  twisting homomorphism for $\Lambda SL(3,\mathbb{C})_{\sigma}$ is denoted by $\hat{\sigma}$ in \eqref{sigma:twist} and  we also denote the differential $d\hat{\sigma}$ for $\Lambda sl(3,\mathbb{C})_{\sigma}$  by $\hat{\sigma}$.

Consider \begin{equation}\label{eq:ode-z_0}
\begin{cases}
dL=L\eta,\\
L(z_0,\lambda)=I,
\end{cases}
\end{equation}
for any (say  real positive)
$z_0 \in \mathcal{D}_R\backslash\{0\}$, 
and restrict the domain of definition to $\mathcal{D}' = \mathcal{D}_R \setminus \R_{\leq 0}.$
Now this domain is simply connected and we can solve the ode for $L$ in our twisted Banach Lie group.

Notice that $\eta$ lies in $\Lambda sl(3,\mathbb{C})_{\sigma}$, thus the solution $L$ also lies in $\Lambda SL(3,\mathbb{C})_{\sigma}$, but, so far, only for the cut domain.
 
 Since $e^{\ln z(-iD(\lambda))}P(z,\lambda)$ is  also a fundamental solution to the equation $dL=L\eta$,  in $\mathcal{D}'$ we obtain the following description for  the solution to \eqref{eq:ode-z_0}:
 $$L(z,\lambda)=B(\lambda) e^{\ln z(-iD(\lambda))}P(z,\lambda).$$

Define $\hat{L}:=\hat{\sigma}(L)$, $\hat{B}:=\hat{\sigma}(B)$, and $\hat{P}:=\hat{\sigma}(P)$. 
Then we have
$$\hat{L}(z,\lambda)=\hat{B}(\lambda) e^{\ln z (-iD(\lambda))} \hat{P}(z,\lambda),$$
and $L(z,\lambda)=\hat{L}(z,\lambda)$. 
Thus we derive
$$e^{\ln z (i D(\lambda))} \hat{B}^{-1}(\lambda) B(\lambda) e^{\ln z (-iD(\lambda))}=\hat{P}(z,\lambda) P^{-1}(z,\lambda).$$

Since the right side is holomorphic in $\mathfrak{D}'$ with limit $I$ at $z = 0$ it follows that $\mathrm{ad} D(\lambda) (\hat{B}^{-1}(\lambda)B(\lambda))=0$ holds.
Therefore the left side of the last equation is equal to  $\hat{B}^{-1}(\lambda) B(\lambda)$ and 
  constant in $z$.

  Notice that $\hat{P}(0, \lambda)P^{-1}(0, \lambda)=I$, hence $\hat{B}^{-1}(\lambda) B(\lambda)=I$ and $\hat{B}(\lambda)=B(\lambda)$ follows. Therefore, $\hat{P}(z,\lambda)=P(z,\lambda)$
  and it follows that $H(z,\lambda)=e^{\ln z (-iD(\lambda))}P(z,\lambda)$ is in $\Lambda SL(3,\mathbb{C})_{\sigma}$. 
 \end{proof}

\begin{theorem}\label{cor:1} 
All the  $*$-perturbed Delaunay potentials $\eta=\frac{-i D(\lambda)}{z}+\eta_+(z,\lambda)$ defined on $\mathcal{D}_R\backslash\{0\}$ 
yield, by an application of Remark \ref{general DPW}  to $H$,  full globally horizontally liftable minimal Lagrangian 
cylinders  $f$ satisfying 
$$f(z + 2 \pi,\lambda)=\chi(\lambda)f(z,\lambda),$$
with $\chi(\lambda)=e^{2\pi D(\lambda)}$.

In particular, the monodromy matrix $\chi(\lambda)$ is unitary for all $\lambda \in S^1$ and satisfies
$\chi(\lambda= 1)=I$, since
the Delaunay matrix $-i D(\lambda=1)$ has, by definition,  only integer eigenvalues.
\end{theorem}

\begin{corollary}Under the assumptions of Theorem  \ref{cor:1}, a $*$-perturbed Delaunay potential associated to the Delaunay matrix $D(\lambda)$ yields a minimal Lagrangian cylinder for $\lambda=1$.
This last condition is satisfied since $\exp(-2\pi i D(\lambda=1))=I$.
\end{corollary}

%%%%%%%%%%%%%%%%%%%%%%%%

\subsection{$*$-perturbed  minimal Lagrangian cylinders in \texorpdfstring{$\C P^2$}{} approximate Delaunay cylinders.}

When considering a perturbed Delaunay potential  $\eta(z,\lambda) = \frac{1}{z} (-iD(\lambda)) dz + \eta_0(z,\lambda) dz$, defined on a (possibly small) disk $\mathcal{D}_R$
about $z = 0$, one wonders how the actual Delaunay surface, defined by the  Delaunay potential 
$\eta_D (z, \lambda)=  \frac{1}{z} (-iD(\lambda)) dz$, relates to the surface defined by $\eta$.
 It is needless to say that one expects the perturbed Delaunay immersion to ``approximate" the  Delaunay immersion. For the case of constant mean curvature surfaces in $\R^3,$ this has been proven in 
 \cite{DW08} , Theorem 4.8.1 (See also \cite{KRS} 
 and   \cite{R}).
 Since our setting is different in several ways, we will give a proof of such an approximation result,
 but follow \cite{DW08}  mutatis mutandis. We will restrict to   $*$-perturbed Delaunay potentials 
 here, and leave the most general case to another publication, since in full generality one needs to use 
 the $r$-loop formalism which has not been introduced to the study of minimal Lagrangian surfaces so far.
 
 \begin{theorem}[Asymptotic behaviour]\label{Thm:asymp}
Let $\eta(z,\lambda) = \frac{1}{z} (-iD(\lambda)) dz + \eta_+(z,\lambda)$ be a 
 $*$-perturbed Delaunay potential. 
Let $H$ denote a solution to the ODE $dH = H \eta$ satisfying the  ``ZAP''  representation
$$ H = e^{\ln z (-iD(\lambda))} P(z,\lambda)$$
stated in Theorem \ref{thm:Sch-Sch}. 

Let  $H=\Phi V_+$  and  $e^{\ln z (-iD(\lambda))}=\Psi W_+$    be  unique 
 Iwasawa decompositions. Then we obtain for all sufficiently small $z \neq 0$:
 $$||\Phi-\Psi|| \leq C_1 |z|$$
 for some $0 < C_1$. Here $||\cdot ||$ is the weighted Wiener norm defined in the Appendix.
  
Finally, considering the actual minimal Lagrangian surfaces
$ f = [\Phi.e_3]$ and $f_D = [\Psi.e_3]$ we obtain the analogous estimates. In particular, the minimal Lagrangian cylinder $f$ defined from $\eta$  at $\lambda = 1$ approximates the equivariant minimal Lagrangian cylinder $f_D$ when $z$ approaches $z =0$. 
 \end{theorem}

\begin{proof}
Using  $H=\Phi V_+ =  e^{\ln z (-iD(\lambda))} P(z,\lambda),$ and 
$e^{\ln z (-iD(\lambda))}=\Psi W_+$ 
as  stated above, it is straightforward to verify
\begin{equation}\label{eq:4.8.2}
W_+PW_+^{-1}=\Psi^{-1}\Phi \cdot V_+W_+^{-1}.
\end{equation}
Note, since $V_+$ and $W_+$ have leading terms $V_0$ and $W_0$ with positive real diagonal entries, the decomposition above is the unique 
Iwasawa decomposition of $W_+PW_+^{-1}$.
 
Let  $d_1, d_2$ and $d_3$ denote the (real) eigenvalues of $-i D(\lambda).$

Then $e^{\ln z (-iD)}=z^{d_1}L_1+z^{d_2}L_2+z^{d_3}L_3$  with  Hermitian projections $L_j$ for $j=1,2,3$. 
It follows that 
\begin{equation}\label{eq:W_+}
W_+=\Psi^{-1}e^{\ln z (-iD)}:=z^{d_1}W_1+z^{d_2}W_2+z^{d_3}W_3,
\end{equation}
 where $W_j=\Psi^{-1}L_j$ are bounded for all values of $z\in\mathcal{D}_R$ and $\lambda$ in an annulus of $S^1$ for $j=1, 2, 3$. 
 
 Similarly, 
\begin{equation}\label{eq:W_+^-1}
W_+^{-1}=z^{-d_1} \hat{W}_1+z^{-d_2} \hat{W}_2+z^{-d_3} \hat{W}_3, 
\end{equation}
where $\hat{W}_j=L_j\Psi $ are bounded for all $z\in \mathcal{D}_R$ and $\lambda$ in the annulus of  $S^1$ for $j=1,2,3$. 
Substituting \eqref{eq:W_+}, \eqref{eq:W_+^-1} and  using the Taylor expansion of the 
holomorphic function $P(z, \lambda)$ at $z =0$
\begin{equation}
P(z,\lambda)=I+ \sum_{k=N}^{\infty }z^k P_k(\lambda),
\end{equation}
we obtain 
\begin{align*}
W_+PW_+^{-1}&=I+\sum_{j=N}^{\infty} ( z^j R_j + z^{d_1-d_2+j}Q_{1j2}+z^{d_2-d_1+j}Q_{2j1} 
+z^{d_2-d_3+j}Q_{2j3} \\
&\quad\quad +z^{d_3-d_2+j}Q_{3j2} +  z^{d_1-d_3+j}Q_{1j3}+ z^{d_3-d_1+j}Q_{3j1}),
\end{align*}
where $R_j=\sum_{\tau=1}^{3} W_{\tau}P_j\hat{W}_{\tau}$, $Q_{\tau j\mu}=W_{\tau}P_j \hat{W}_{\mu}$, with $\tau,\mu=1,2,3$, are bounded for all $z$ in a smaller disk around $0$ and $\lambda$ in a possibly smaller annulus of $S^1$.

Since we have assumed that $\eta$ is a $*$-perturbed Delaunay potential, we know that
all expressions $d_{\tau}-d_{\mu}+j$ are positive for all $\lambda$  in an annulus of $S^1$. As a consequence, $W_+PW_+^{-1}$ converges to $I$ as $z$ goes to $0$. Since the Iwasawa decomposition (of $W_+PW_+^{-1}$) is a real analytic map preserving $I$,  we also obtain that
$ \Psi^{-1}\Phi$ approaches  $I$ if $z$ goes to $0$.

Since $||\Psi||=1$, we have
$$||\Phi-\Psi||=||\Psi (\Psi^{-1} \Phi - I) ||\leq ||\Psi|| ||\Psi^{-1} \Phi - I ||
\rightarrow 0.$$

Note that this implies in particular, that for each $\lambda$  in a sufficiently small annulus of  $S^1$,  we have $\Phi(z, \lambda)-\Psi (z, \lambda) \rightarrow 0$ as $z \rightarrow 0$.

Finally, we obtain $\Phi e_3 -\Psi e_3\rightarrow 0$ in $\mathbb{C}^3$ as $|z|$ goes to $0$ and $|\lambda-1|< s$ small enough.  That is to say, we have the immersion $f=[\Phi e_3]$ approximates to the Delaunay immersion $f_D=[\Psi e_3]$ in $\mathbb{C}P^2$. 
\end{proof}
%%%%%%%%%%%%%%%%%%%%%%%%
\subsection{Almost-uniqueness of $*$-perturbed minimal Lagrangian cylinders in \texorpdfstring{$\C P^2$}{}}
\subsubsection{\textbf{Basic setup}}
In the above construction of $*$-perturbed minimal Lagrangian cylinders, we do not use any base point. Hence $H$ is not normalized at any $z\in \C\backslash \{0\}$. Nevertheless, one can apply the usual procedure as outlined in Section 2.  However, a priori, one looses uniqueness statements between minimal Lagrangian immersions and certain potentials.
 Fortunately, it turns out that for  $*$-perturbed minimal Lagrangian cylinders in $\C P^2$ and potentials as considered above, the relation is still ``basically unique", as explained in more detail at the end of this section.

Let's start by describing the procedure in the case under consideration.

Assume that  $\eta$ and $\eta^{\#}$ are two  $*$-perturbed Delaunay potentials given by
\begin{equation}
\begin{aligned}
\eta(z,\lambda)&=\frac{-iD(\lambda)}{z}dz+\eta_+(z,\lambda),\\
\eta^{\#}(z,\lambda)&=\frac{-iD^{\#}(\lambda)}{z} dz+\eta_+^{\#}(z,\lambda),
\end{aligned} 
\end{equation}
defined on $\mathcal{D}_R\backslash\{0\}$,  and with

 $$\eta_+(z,\lambda)=
\sum_{k=N}^{\infty}  \mathring{\eta}_{+,k}(\lambda)  z^k dz \hspace{2mm}\mbox{and} \hspace{2mm} 
\eta_+^{\#}(z,\lambda)=\sum_{k=N^{\#}}^{\infty}  \mathring{\eta}_{+,k}^{\#}(\lambda)  z^k dz$$

 for some especially  chosen $N>0$ and $N^{\#}>0$.

 Let $H$ denote the solution to $dH=H\eta$ given 
\begin{equation} \label{defH}
H(z,\lambda)=e^{\ln z (-iD(\lambda))}P(z,\lambda),
\end{equation}
 where $P(z,\lambda)$ is holomorphic for $z\in \mathcal{D}_R$ and $P(0,\lambda)=I$ for any $\lambda\in \mathbb{C}^*$. 
 
 Starting from $H$, we perform the (unqiue) Iwasawa decomposition 
 \begin{equation}\label{eq:H-6}
 H(z,\lambda)=\mathbb{F}(z,\bar{z},\lambda) V_+(z,\bar{z},\lambda),
 \end{equation}
 where $\mathbb{F}\in \Lambda SU(3)_{\sigma}$ and $V_+\in \Lambda^{+} SL(3,\mathbb{C})_{\sigma}$. 
 
From this one obtains the ``associated family" of minimal Lagrangian surfaces in $\C P^2$ given by 
$$f(z,\bar{z}, \lambda)=\mathbb{F}(z,\bar{z},\lambda) e_3 \mod  U(1).$$
In the same way we introduce $H^{\#}$, $\mathbb{F}^{\#}$ and  $f^{\#}$ relative to $\eta^{\#}$. 
For $\lambda=1$, we obtain
 $f(z,\bar{z}), f^{\#}(z,\bar{z}): \mathcal{D}_R\backslash\{0\}\rightarrow \C P^2$, the corresponding $*$-perturbed minimal Lagrangian cylinders with respect to $\eta$ and $\eta^{\#}$, respectively.

\begin{definition}
We say that $f$ and $f^{\#}$ as constructed above are \emph{ equivalent}, if there exists some 
$\gamma \in \mathrm{Aut}(\mathcal{D}_R \backslash \{0\}),$ where $0 < R < \infty,$ and some $W \in SU(3)$ such that
\begin{equation}\label{eq:sym0}
f^{\#}(z,\bar{z})=W \gamma^*f( z, \bar{z}),
\end{equation}
for all $z\in \mathcal{D}_R\backslash\{0\}$.
\end{definition}

\begin{remark}
(a) In the case of our domain of definition we have
$$\gamma: z \mapsto e^{is}z,$$
for some $s \in \mathbb{R}$.

(b) If $R=\infty$ is permitted, then also $z \mapsto \frac{b}{z}$, $b\in\C$, is an automorphism. But switching $z=0$ and $z=\infty$ means that an end at $z=0$ is changed to an end at $z=\infty$. So in this paper,  only 
$\gamma\cdot z=e^{is} z$ will be considered for some (fixed) $s \in \R$.

(c) The relation (\ref{eq:sym0}) basically states that two $*$-perturbed minimal Lagrangian cylinders 
 $f(z,\bar{z}) : \mathcal{D}_R\backslash\{0\}\rightarrow \C P^2$ 
 and  $f^{\#}(z,\bar{z}): \mathcal{D}_R^{\#} \backslash\{0\}\rightarrow \C P^2$, 
 with respect to the potentials $\eta$ and $\eta^{\#}$, respectively, are isometrically related as expressed by the diagram
  \[
  \begin{diagram}
    \node{ \mathcal{D}_R \backslash\{0\}}  \arrow{r,l}{f} \arrow{s,l}{\gamma} \node{\C P^2} \arrow{s,r}{W}  \\
    \node{ \mathcal{D}_R^{\#} \backslash\{0\}} \arrow{r,l}{f^{\#}}\node{\C P^2} \\
 \end{diagram}
\]
While here the commutativity of the diagram is an assumption, we expect that, as for other surface classes, 
 in case the induced metrics are complete, the isometry $W$ transporting  
$f$ to $f^{\#}$ does imply the existence of some automorphism 
$\gamma:  \mathcal{D}_R \backslash\{0\} \rightarrow  \mathcal{D}_R^{\#} \backslash\{0\}$
which makes the diagram complete.
\end{remark}

It follows from \eqref{eq:sym0} that there exists  $W(\lambda)\in \Lambda SU(3)_{\sigma}$ such that
\begin{equation}\label{eq:sym}
f^{\#}(z,\bar{z},\lambda)=W(\lambda) \gamma^*f( z, \bar{z}, \lambda).
\end{equation} 
Thus on the level of extended frames it holds
\begin{equation}
\mathbb{F}^{\#}(z,\bar{z},\lambda)=W(\lambda) \gamma^*\mathbb{F}(z,\bar{z},\lambda) K(z,\bar{z}),
\end{equation}
where $K(z,\bar{z})\in \Lambda^{+}SL(3,\mathbb{C})_{\sigma}$ is diagonal and independent on $\lambda$.
Using \eqref{eq:H-6} we infer
\begin{equation}\label{eq:H-sharp}
H^{\#}(z,\lambda)=W(\lambda) H(\gamma\cdot z, \lambda) L_+(z, \bar{z}, \lambda),
\end{equation}
with $L_+(z,\bar{z},\lambda)\in \Lambda^{+}SL(3,\C)_{\sigma}$.
Computing the Maurer-Cartan form of both sides of this equation yields
\begin{equation}\label{eq:gauge}
\eta^{\#}(z,\lambda)=(\gamma^*\eta)(z,\lambda)\# L_{+}(z,\bar{z},\lambda),
\end{equation}
where $\#$ means gauge transformation by $L_{+}(z,\bar{z},\lambda)$.

\subsubsection{\textbf{Getting rid of \texorpdfstring{$\gamma^*$}{}}}

Starting from $\eta^{\#}$ and 
\begin{equation}\label{eq:check-eta}
\check{\eta}:=\gamma^*\eta=-i\frac{D(\lambda)}{z}dz+\gamma^*\eta_{+}(z,\lambda),
\end{equation}
we obtain $H^{\#}(z,\lambda)$ and
\begin{equation}\label{eq:H_check}
\check{H}(z,\lambda)=e^{\ln z (-iD(\lambda))}\check{P}(z,\lambda),
\end{equation}
where $\check{P}(z,\lambda)$ is holomorphic for $z\in \mathcal{D}_R$ and $\check{P}(0,\lambda)=I$ for any $\lambda\in \mathbb{C}^*$.
Moreover, due to \eqref{eq:check-eta} and 
$$\check{\eta}=\check{H}^{-1} d\check{H}=\check{P}^{-1}\frac{-iD(\lambda)}{z} dz\check{P}+\check{P}^{-1}d\check{P},$$
we obtain (also using $\gamma (0) = 0$)
$$\check{P}(z,\lambda)=\gamma^* P(z,\lambda).$$
Thus,
$$\gamma^*H(z,\lambda)=e^{tD(\lambda)} \check{H}(z,\lambda),$$
and \eqref{eq:H-sharp} reads that
$$H^{\#}(z,\lambda)=W(\lambda) e^{tD(\lambda)} \check{H}(z,\lambda) L_{+}(z,\bar{z},\lambda).$$

Set $\check{W}(\lambda)=W(\lambda) e^{tD(\lambda)}\in \Lambda SU(3)_{\sigma}$. This implies 
\begin{equation}
H^{\#}(z,\lambda)=\check{W}(\lambda)\check{H}(z,\lambda)L_{+}(z,\bar{z},\lambda).
\end{equation}
From here we get $\mathbb{F}^{\#}$ and $\check{\mathbb{F}}$ as usual. Hence
\begin{equation}
\mathbb{F}^{\#}(z,\bar{z},\lambda)=\check{W}(\lambda)\check{\mathbb{F}}(z,\bar{z},\lambda)\check{K}(z,\bar{z}),
\end{equation}
and 
\begin{equation}
f^{\#}(z,\bar{z},\lambda)=\check{W}(\lambda)\check{f}(z,\bar{z},\lambda).
\end{equation}

Altogether we have shown

\begin{theorem} \label{reduce to gamma = id}
Let $f$ and $f^{\#}$ be equivalent $*$-perturbed minimal Lagrangian cylinders. Then the corresponding elements
of the associated family satisfy  (by equation (\ref{eq:sym}) )
\begin{equation}
f^{\#}(z,\bar{z},\lambda)=W(\lambda) \gamma^*f( z, \bar{z}, \lambda),
\end{equation} 
and thus are equivalent for corresponding $\lambda$.
\newline Put  (see equation \eqref{eq:check-eta})
\begin{equation*}
\check{\eta}:=\gamma^*\eta=-i\frac{D(\lambda)}{z}dz+\gamma^*\eta_{+}(z,\lambda),
\end{equation*}
Then the corresponding  $*$-perturbed minimal Lagrangian cylinder $\check{f}$ defined from $\check{\eta}$
satisfies 
\begin{equation} \label{nogamma}
f^{\#}(z,\bar{z},\lambda)=\check{W}(\lambda)\check{f}(z,\bar{z},\lambda),
\end{equation}
where  $\check{W}(\lambda)=W(\lambda) e^{tD(\lambda)}\in \Lambda SU(3)_{\sigma},$ where $\gamma \cdot z = e^{it}z.$

\end{theorem}

Next we will investigate, what equation (\ref{nogamma}) means for $f$ and $\check{f}$.

%%%%%%%%%%
\subsubsection{\textbf{The case where $\gamma=\mathrm{id}$: implications for $D^{\#}(\lambda)$ and $D(\lambda)$}}
%%%%%%%%%
We can therefore simplify the setting by putting $\gamma=\mathrm{id}$ in \eqref{eq:sym}, i.e.,
\begin{equation}\label{eq:sym1}
f^{\#}(z,\bar{z},\lambda)=W(\lambda)f(z,\bar{z},\lambda),
\end{equation}
with $W(\lambda)\in \Lambda SU(3)$. Thus 
$\mathbb{F}^{\#}(z,\bar{z},\lambda)=W(\lambda)\mathbb{F}(z,\bar{z},\lambda)K(z,\bar{z})$.

Now we use Theorem \ref{Thm:asymp} and its proof. We write the extended frames for $f$ and $f^{\#}$ generated by the $*$-perturbed Delaunay potentials $\eta$ and $\eta^{\#}$, respectively,  by
\begin{equation}
\begin{aligned}
\mathbb{F}&=\Psi B, \text{ with } B\rightarrow I \text{ as } z\rightarrow 0,\\
\mathbb{F}^{\#}&=\Psi^{\#} B^{\#}, \text{ with } B^{\#}\rightarrow I \text{ as } z\rightarrow 0.
\end{aligned}
\end{equation}
 
 Taking \eqref{eq:sym1} into account, we get  
 $\Psi^{\#} B^{\#}=W(\lambda)\Psi B K$, hence 
 \begin{equation}\label{eq:3.3}
 \Psi^{\#}(z,\bar{z},\lambda)=W(\lambda) \Psi(z,\bar{z},\lambda) S(z,\bar{z},\lambda),
 \end{equation}
 where $S(z,\bar{z},\lambda):=BK{B^{\#}}^{-1} \in \Lambda SU(3)_{\sigma}$ and $S\rightarrow K(0,0)$ as $z\rightarrow 0$. 

Since $D^{\#}$ and $D$ generate  minimal Lagrangian Delaunay surfaces on $\mathcal{D}_R\backslash\{0\}$ with extended frames $\Psi^{\#}$ and $\Psi$, we  know without loss of generality $\Psi(1,1,\lambda)=I$ and 
$\Psi^{\#}(1,1,\lambda)=I$. 
Note that given an automorphism $\gamma_t: z\mapsto e^{it}z$ on $\mathcal{D}\backslash\{0\}$ we obtain
\begin{eqnarray}
(\gamma_t^*\Psi)(z,\bar{z},\lambda)&=&\chi(t,\lambda) \Psi(z,\bar{z},\lambda),\\
(\gamma_t^*\Psi^{\#})(z,\bar{z},\lambda)&=&\chi^{\#}(t,\lambda) \Psi^{\#}(z,\bar{z},\lambda).
\end{eqnarray}
From \eqref{eq:3.3} we obtain
\begin{equation}\label{eq:3.6}
\chi^{\#}(t,\lambda)\Psi^{\#}(z,\bar{z},\lambda)=W(\lambda)\chi(t,\lambda)\Psi(z,\bar{z},\lambda)\gamma_t^*S(z,\bar{z},\lambda).
\end{equation}
Since $\chi^{\#}$, $\chi$, $\Psi^{\#}$, $\Psi$ and $W$ are unitary, $\gamma_t^* S(z,\bar{z},\lambda)$, also $S(z,\bar{z},\lambda)$, are unitary and converge to $K(0,0)$ as $z\rightarrow 0$.
From \eqref{eq:3.6} and \eqref{eq:3.3} we derive
$$\chi^{\#}(t,\lambda)W(\lambda)\Psi(z,\bar{z},\lambda) S(z,\bar{z},\lambda)=W(\lambda)\chi(t,\lambda)\Psi(z,\bar{z},\lambda)\gamma_t^*S(z,\bar{z},\lambda),$$
which implies
$${\chi(t,\lambda)}^{-1} {W(\lambda)}^{-1}\chi^{\#}(t,\lambda)W(\lambda)=
\Psi(z,\bar{z},\lambda)(\gamma_t^*S(z,\bar{z},\lambda)  S(z,\bar{z},\lambda)^{-1})\Psi(z,\bar{z},\lambda)^{-1},$$
where the left side of  is independent of $z$ and the right side tends to $I$ as $z\rightarrow 0$. 
This leads to
\begin{eqnarray}
\chi^{\#}(t,\lambda)&=&W(\lambda)\chi(t,\lambda)W(\lambda)^{-1},\label{eq:3.7}\\
\gamma_t^*S(z,\bar{z},\lambda)&=&S(z,\bar{z},\lambda),
\end{eqnarray}
for all $t\in \mathbb{R}$.
In particular, $S$ only depends on the radius, not on the angle. 
Consequently, equivariant minimal Lagrangian immersions $f_{D^{\#}}$ and $f_{D}$ generated from $D^{\#}(\lambda)$ and $D(\lambda)$ satisfy
\begin{equation}\label{eq:3.9}
f_{D^{\#}}(z,\bar{z},\lambda)=W(\lambda) f_{D}(z,\bar{z},\lambda).
\end{equation}
And conversely, \eqref{eq:3.9} implies \eqref{eq:3.7}.

Since $W(\lambda)$ induces an isometry of $\mathbb{C} P^2$, the metrics of  $f_{D^{\#}}$ and $f_{D}$ coincide. 
We can thus assume without loss of generality that $D^{\#}(\lambda)$ and $D(\lambda)$ both have vanishing diagonal (choice of $x=0$ in \cite{DoMaNL}). Moreover, since $D(\lambda)=\Omega(0)$ (and same with $D^{\#}(\lambda)$) by Section 5.1 of \cite{DoMaNL} , we have for the defining coefficients $a^{\#}=a$ in \eqref{eq:Delaunay}.

On the other hand, the horizontal lifts $\mathfrak{f}_{D^{\#}}$ and   $\mathfrak{f}_D$ satisfy
$$\mathfrak{f}_{D^{\#}}(z,\bar{z},\lambda)=W(\lambda) \mathfrak{f}_{D}(z,\bar{z},\lambda).$$
Hence for the cubic forms we obtain
$$\psi_{D^{\#}}(z,\bar{z},\lambda)=\psi_{D}(z,\bar{z},\lambda)$$
by the defintion \eqref{eq:phipsi}.  

As a consequence,
\begin{theorem}\label{thm:3.10}
For two equivalent $*$-perturbed equivariant minimal Lagrangian cylinders, we obtain
\begin{equation}\label{eq:3.10}
\begin{aligned}
& D^{\#}(\lambda)=D(\lambda),\\
& [W(\lambda), D(\lambda)]=0. 
\end{aligned}
\end{equation}
\end{theorem} 

%%%%%%%%
\subsubsection{\textbf{The case where $\gamma=\mathrm{id}$: classification}}
%%%%%%%%

Using what we proved in the last subsection, we recall:
\begin{equation*}
\begin{aligned}
f^{\#}(z,\bar{z},\lambda)&=W(\lambda) f(z,\bar{z},\lambda),\\
\mathbb{F}^{\#}(z,\bar{z},\lambda)&= W(\lambda) \mathbb{F}(z,\bar{z},\lambda)K(z,\bar{z}),\\
H^{\#}(z,\lambda)&= W(\lambda) H(z,\lambda)V_+(z,\bar{z},\lambda) K(z,\bar{z})V_+^{\#}(z,\bar{z},\lambda).
\end{aligned}
\end{equation*}
The last equation  can be written as
$$e^{\ln z(-iD^{\#}(\lambda))}P^{\#}(z,\lambda)=W(\lambda) e^{\ln z(-iD(\lambda))} P(z,\lambda) L_+(z,\lambda),$$
where $L_+(z,\lambda)=V_+(z,\bar{z},\lambda) K(z,\bar{z})V_+^{\#}(z,\bar{z},\lambda)$.
By Theorem \ref{thm:3.10}, this is equivalent to 
\begin{equation}\label{eq:4.1}
P^{\#}(z,\lambda)=W(\lambda)P(z,\lambda)L_+(z,\lambda).
\end{equation}

This implies that $L_+(z,\lambda)$ is holomorphic at $z=0$. Since $P(0,\lambda)=I$ and $P^{\#}(0,\lambda)=I$, we obtain
\begin{equation}\label{eq:4.2}
L_+(0,\lambda)=W(\lambda).
\end{equation}
As a consequence, 
\begin{equation}\label{eq:4.3}
\begin{aligned}
& L_+(0,\lambda)\in \Lambda^+ SL(3,\mathbb{C})\cap \Lambda SU(3)_{\sigma},\\
&  L_+(0,\lambda) \text{ and } W(\lambda)  \text{ are diagonal, unitary and } \lambda\text{-independent matrices}.
 \end{aligned}
 \end{equation}
 
 Applying $[W(\lambda), D(\lambda)]=0$ in \eqref{eq:3.10} and \eqref{eq:4.3}, we infer $W(\lambda)=\mathrm{diag} (u,v,w)$. Thus the commutation property for 
 $$D(\lambda)=\begin{pmatrix}
0&d_{12}(\lambda)&d_{13}(\lambda)\\
d_{21}(\lambda)&0&d_{23}(\lambda)\\
d_{31}(\lambda)&d_{32}(\lambda)&0
 \end{pmatrix}$$
 gives
 $$\begin{pmatrix}
0&u d_{12}(\lambda)&u d_{13}(\lambda)\\
v d_{21}(\lambda)&0&v d_{23}(\lambda)\\
w d_{31}(\lambda)&w d_{32}(\lambda)&0
 \end{pmatrix}=\begin{pmatrix}
0&v d_{12}(\lambda)&w d_{13}(\lambda)\\
u d_{21}(\lambda)&0&w d_{23}(\lambda)\\
u d_{31}(\lambda)&v d_{32}(\lambda)&0
 \end{pmatrix}$$
 for all $\lambda\in S^1$.
 
 By our assumptions on $D(\lambda)$, all $d_{ij}(\lambda)$ $(i\neq j)$ are generically non-vanishing. Thus we obtain
 \begin{equation}\label{eq:4.4}
 u=v=w.
 \end{equation}
 In particular, $W(\lambda) = \mathrm{id}$ when considered as an isometry of $\C P^2$.

 As a consequence,
 $$[f^{\#}(z,\bar{z},\lambda)]=[f(z,\bar{z},\lambda)].$$
 
 Moreover, \eqref{eq:4.1}, \eqref{eq:4.2} and \eqref{eq:4.4} imply
 $$P^{\#}(z,\lambda)=P(z,\lambda).$$
Hence, $P^{-1}dP=(P^{\#})^{-1}dP^{\#}$ and since $D(\lambda)=D^{\#}(\lambda)$, we infer
$$\eta_{+}^{\#}(z,\lambda)=\eta_{+}(z,\lambda).$$

All together, we have shown
\begin{theorem} \label{classification if gamma = id}
If $\eta$ and $\eta^{\#}$ are $*$-perturbed equivariant potentials inducing the same $*$-perturbed equivariant cylinders, in the sense of \eqref{eq:sym1}, by the procedure explained
in this section, then $\eta(z,\lambda)=\eta^{\#}(z,\lambda)$. 
\end{theorem}

%%%%%%%%%%%%%%%%%%%%%
\subsection{Classification of equivalent  $*$-perturbed equivariant cylinders}

Putting together Theorem \ref{reduce to gamma = id} and Theorem \ref{classification if gamma = id} we obtain

\begin{theorem} [Classification of equivalent  $*$-perturbed equivariant cylinders]
Let $f$ and $f^{\#}$ be  $*$-perturbed minimal Lagrangian cylinders derived from the 
 $*$-perturbed Delaunay potentials    $\eta$ and $\eta^{\#}$ given by
\begin{equation}
\begin{aligned}
\eta(z,\lambda)&=\frac{-iD(\lambda)}{z}dz+\eta_+(z,\lambda),\\
\eta^{\#}(z,\lambda)&=\frac{-iD^{\#}(\lambda)}{z} dz+\eta_+^{\#}(z,\lambda),
\end{aligned} 
\end{equation}
defined on $\mathcal{D}_R\backslash\{0\}$.
\newline If $f$ and $f^{\#}$ are equivalent, then $f^{\#}$ is the  $*$-perturbed equivariant cylinder defined 
from $\gamma^* \eta =  -i\frac{D(\lambda)}{z}dz+\gamma^*\eta_{+}(z,\lambda),$ where $\gamma \in S^1.$
\newline In particular, all  $*$-perturbed equivariant cylinders which are equivalent to a given  $*$-perturbed equivariant cylinder $f$ form the circle of potentials $\gamma^*\eta$.
\end{theorem}

\begin{proof} 
Let $f$ and $f^{\#}$ be equivalent $*$-perturbed minimal Lagrangian cylinders. Then the corresponding elements
of the associated family satisfy  (by equation (\ref{eq:sym}) )
\begin{equation}
f^{\#}(z,\bar{z},\lambda)=W(\lambda) \gamma^*f( z, \bar{z}, \lambda),
\end{equation} 
and thus are equivalent for corresponding $\lambda$.
\newline Putting  (see equation (\ref{eq:check-eta}))
\begin{equation*}
\check{\eta}:=\gamma^*\eta=-i\frac{D(\lambda)}{z}dz+\gamma^*\eta_{+}(z,\lambda),
\end{equation*}
Then by Theorem \ref{reduce to gamma = id}   the corresponding  $*$-perturbed minimal Lagrangian cylinder $\check{f}$ defined from $\check{\eta}$
satisfies 
\begin{equation} \label{nogamma}
f^{\#}(z,\bar{z},\lambda)=\check{W}(\lambda)\check{f}(z,\bar{z},\lambda),
\end{equation}
where  $\check{W}(\lambda)=W(\lambda) e^{tD(\lambda)}\in \Lambda SU(3)_{\sigma}$. 
Now Theorem \ref{classification if gamma = id} finishes the proof.
\end{proof} 

%%%%%%%%%%%%%
\section{Appendix A: About norms}
%%%%%%%%%%%%%%%%%%%%%
\subsection{Introduction}  In this paper we will exclusively use weighted Wiener norms,
which occur in many places. For the purposes of this study, we mainly refer to 
the appendix of \cite{GoWa}.
The statements therein primarily focus on scalar algebras, as well as on 
untwisted matrix algebras and untwisted  matrix Lie algebras,
which naturally form Banach (resp. Banach Lie) algebras.

The algebras and Lie algebras considered in this paper, however, are generally twisted.
Nevertheless,  due to the fact that the twisted subalgebras of these Banach algebras are closed and have a closed complement in the untwisted ones, the twisted algebras also inherit the structure of Banach algebras.

The norms for the matrix algebras, while expressed simply in terms of coefficients in the untwisted setting, exhibit different explicit expressions in the twisted case due to coinciding or related coefficients.
In the formulas below, we will omit explicit expressions and directly apply the formulas/relations provided for the untwisted case.

%%%%%%%%%%%%%%%%%%%%%%%%%%%%%%%
\subsection {Scalar weighted Wiener norms and algebras}\label{sec:6.7}

\vspace{3mm}

A function $\omega : \mathbb{Z} \rightarrow (0,\infty)$ is called a \emph{weight}, if 
$\omega (k + l) \leq \omega (k) \omega(l), \hspace{2mm} k,l \in \mathbb{Z}$.
Note that this implies $1 \leq \omega(0)$.

A weight $\omega$ is called \emph{symmetric}, if 
$$ \omega(-k) = \omega (k), \hspace{2mm} k \in \mathbb{Z}.$$
In this case we have $1 \leq \omega(0)^\frac{1}{2} \leq \omega (k)$ for all $k \in \mathbb{Z}$.

For a symmetric weight $\omega$ we define
$$ A_\omega = \{f: S^1 \rightarrow \C, \lambda \mapsto \sum_{ n \in \mathbb{Z}} a_n \lambda^n, ||f||_\omega < \infty\},$$
where 
$$||f||_\omega = \sum_{n \in \mathbb{Z}} |a_n| \cdot \omega (n).$$

It is sometimes useful to use
$$ A_\omega^+ = \{f: S^1 \rightarrow \C, \lambda \mapsto \sum_{ n \in \mathbb{Z}, n > 0} a_n \lambda^n,
 ||f||_\omega < \infty \}$$
 and analogously $A_\omega^-,$ since then we can write 
 $$ A_\omega = A_\omega^+ \oplus \mathbb{C}\cdot 1 \oplus  A_\omega^-,$$
 which shows that $A_\omega$ is ``decomposing" in the sense of 
 Gohberg \cite{Gohberg}.

Two classes of symmetric weights are

\begin{enumerate}
\item {\bf Polynomial weights}
$$ \omega_a (k) = (1+  |k| )^a, \text{ for } a \geq 0,$$
\item {\bf Gevrey class weights} 
$$ \omega_{t,s}(k) = \exp(t \cdot |k|^s), \text{ for }   t > 0,\,  0 < s < 1.$$
\end{enumerate}

Both weights satisfy $\omega (0) = 1.$
One easily verifies
\begin{theorem}[cf. Appendix of \cite{GoWa} or Section 1 of \cite{DGS}] \label{Thm:norm-vector} 
For any of the above two  basic weights   (which satisfy $\omega(0) = 1$),  we have  
\begin{enumerate}
\item $|| \cdot ||_\omega$  is a norm.
\item $||fg||_\omega \leq ||f||_\omega \cdot ||g||_\omega$ and $||1||_\omega  = 1$.
\item $A_\omega$ is a unital commutative Banach algebra with norm $|| \cdot ||_\omega$ . 
\item If $*$ denotes complex conjugation, then $||f^*||_\omega  = ||f||_\omega $, whence 
$A_\omega$ is a commutative Banach *-algebra.
\item Each of the two types of weights is of {\bf non-analytic type}, meaning
$$ \lim_{n \rightarrow \infty} \omega(n)^{\frac{1}{n}}= 1.$$
\end{enumerate}
\end{theorem}

From here on, we will only consider a norm with the above properties. The specific form of the norm we choose is not relevant for our geometric results.

Note that Proposition A3 in \cite{GoWa} is incorrect; however, it is not utilized in our present work.

%%%%%%%%%%%%%%%%%%%%%%%%%%%%%%%%%%%%%
\subsection{Weighted Wiener Banach matrix $*$-algebras}\label{sec:6.8}

Let $\mathrm{Mat}(n, A_\omega)$ denote the algebra of $n \times n$-matrices with entries in $A_\omega$, where 
 $A_\omega$ is equipped with a norm satisfying the properties listed in Theorem \ref{Thm:norm-vector}.
  
 For an $n$-tuple $\vec{v}$ of elements of  $A_\omega$, i.e. an element  in $A_\omega^n$,
 we set
$$ ||\vec{v}||_\omega  =  \sum_{1 \leq k \leq n} ||v_k||_\omega.$$
This norm defines the structure of a Banach space on $A_\omega^n$.

For $T \in \mathrm{Mat}(n, A_\omega)$,  the induced operator norm is given by \begin{equation}\label{eq: T-norm}
 || T||_\omega = \max_{1 \leq j \leq n} ( \sum_{i = 1}^n ||T_{ij}||_\omega),
 \end{equation}
which defines the structure of a Banach space on $\mathrm{Mat}(n, A_\omega).$

Moreover, we have
\begin{theorem} Using the norms defined above for vectors $\vec{v}$ and matrices $T, B \in \mathrm{Mat}(n, A_\omega)$
yields the following statements:
\begin{enumerate}
\item  $ || T \vec{v}||_\omega \leq  || T||_\omega  || \vec{v}||_\omega$,
\item  $|| TB||_\omega \leq  || T||_\omega  || B||_\omega$,
\item  $|| Id ||_\omega = 1$,
\item $ ||T^*||_\omega = ||T||_\omega,$ where $^*$ denotes complex conjugation.
\item $\mathrm{Mat}(n, A_\omega)$  is a unital 
Banach $*$-algebra relative to the weighted Wiener norm $||\cdot ||_{\omega|}$.
\end{enumerate}
\end{theorem}

%%%%%%%%%%%%%%%%%%%%
\subsection{Banach Lie groups modeled on the weighted Wiener algebra}
Here we refer to \cite{Lang} for  the definition of a Lie subgroup in a Banach Lie group, Lie subalgebra, etc. 
 
In this paper, exclusively subgroups of $SL(n, A_\omega)$ occur. We put

$$ SL(n, A_\omega) = \{g \in \mathrm{Mat}(n, A_\omega): \det g = 1 \}.$$
Clearly, $SL(n, A_\omega)$ is closed in $\mathrm{Mat}(n, A_\omega)$  and thus is a Banach Lie group
 with Lie algebra
$$ sl(n, A_\omega) = \{R \in \mathrm{Mat}(n, A_\omega): \tr R = 0 \}.$$

Moreover, the closed Banach Lie algebra 
$$\{ q \mathrm{Id} : q \in A_\omega \}$$
 is complementary to $ sl(n, A_\omega) $ in $\mathrm{Mat}(n, A_\omega)$.
Thus, $ SL(n, A_\omega)$ is a Banach Lie subgroup of $ GL(n, A_\omega)$.

In addition to the group $SL(n, \C)$ also (real) subgroups $G$, like $SU(3)$, of $SL(n, \C)$ occur.
 To include loop groups taking value in a subgroup $G$ of $SL(n, \C)$
  one introduces the notation:
  $$\Lambda G := \{ g: S^1\rightarrow G  : ||g||_{\omega} < \infty \}$$
  and analogously  for the complex subgroup  
   $G^\C$
  we write
  $$ \Lambda G^\C = \{ g: S^1\rightarrow G^{\mathbb{C}} : \, ||g||_{\omega} < \infty  
 \hspace{2mm} \mbox{and} \hspace{2mm}  \det g(\lambda) = 1\}. $$
 
 Actually, many more different types of loop
 groups are used.
 We will write out the case of twisted loop groups in the following section.
 The untwisted objects are defined in an obvious way analogously.

 %%%%%%%%%%%%%%%%%%
 \section*{Appendix B: Twisted Lie groups and Lie algebras}
 
 In this paper we will consider exclusively \emph{twisted}  loop groups contained in $\Lambda SL(3, \C)$ and 
 their  corresponding loop algebras.
 Moreover, only one specific outer involution $\sigma$ of $ SL(3, A_\omega)$ and the twisting automorphism 
 $\hat{\sigma}$ defined in section \ref{sec:2} 
will occur in this paper (unless specifically something different will be stated explicitly). 

As a  matter of fact, in this paper we consider the simply connected, maximal  real subgroup 
$G = SU(3)$ of $SL(3, \C).$ In particular, we have $G^\C = SL(3, \C).$

We will write down below the definitions of the various loop subgroups used in this paper.

Note that $\hat{\sigma}$ satisfies $\hat{\sigma}(gh)=\hat{\sigma}(g)\hat{\sigma}(h)$. Moreover, $\hat{\sigma}$ can be induced on $\Lambda \mathfrak{g}^{\C}$ with  the same order.
Thus $\Lambda \mathfrak{g}^{\C}$ is the direct sum of the finite order (closed) eigenspaces of $\hat{\sigma}$. 
By $\mathcal{D}$ we
denote the interior of the unit disk  as in section \ref{sec:2} and by $\E=\{\lambda\in S^2 | |\lambda|>1\}$  the exterior of the unit disk. 
Set 
\begin{eqnarray*}
 \Lambda G_{\sigma}^{\mathbb C}&=&\{ g: S^1\rightarrow G^{\C} | \, ||g||_{\omega}<\infty, \, (\hat{\sigma} g)(\lambda)=g(\lambda) \},\\
\Lambda^{+} G^{\mathbb C}_{\sigma}&=&\{g\in \Lambda G^{\mathbb{C}}_{\sigma} | \, g \text{ extends holomorphically to } \mathcal{D}, g(0)\in K^{\mathbb C}\},\\
\Lambda^{+}_B G_{\sigma}^{\mathbb C}&=&\{ g\in \Lambda^{+} G_{\sigma}^{\mathbb C} |\, 
g(0)\in B \},\\
\Lambda^{-} G^{\mathbb C}_{\sigma}&=&\{ g\in \Lambda G^{\mathbb{C}}_{\sigma} | \, g \text{ extends holomorphically to } \mathbb{E}, g(\infty)\in K^{\mathbb C}\},\\
\Lambda^{-}_{*} G^{\mathbb C}_{\sigma}&=&\{ g\in \Lambda^{-} G^{\mathbb C}_{\sigma} | \, g(\infty)=e\},
\end{eqnarray*}
where $K^{\mathbb C} = KB$ is the unique Iwasawa decomposition of $K^{\mathbb C}$ relative to $K$.
(Recall that $K \cong SU(2, \C)$ is compact and  $K^\C$ is a simply connected complex Lie group.)

In this paper, we will always equip $\Lambda G_{\sigma}^{\mathbb C}$ with a weighted Wiener norm.
Unless anything else is stated explicitly, the Wiener topology of absolute convergence of the Fourier coefficients can be chosen.
Actually, due to the fact that the eigenspaces of $\hat{\sigma}$ are closed subspaces of $\Lambda\mathfrak{g}^{\C}$, we see that $\Lambda G_{\sigma}^{\mathbb C}$ is a Lie subgroup of $\Lambda G^{\C}$
(see sections \ref{sec:6.7} and \ref{sec:6.8}). 
Then the group
$\Lambda G^{\mathbb C}_{\sigma}$ becomes a complex Banach Lie group  with Lie algebra
\begin{equation*}
 \Lambda \mathfrak{g}^{\mathbb{C}}_{\sigma}:=\{\xi \in   \Lambda \mathfrak{g}^{\mathbb{C}}|
  (\hat{\sigma}\xi)(\lambda))=\xi(\lambda)\}
\end{equation*}

If $\xi \in \Lambda \mathfrak{g}^{\mathbb C}_{\sigma}$, its Fourier decomposition is
$$\xi=\sum_{l \in \mathbb{Z}} \lambda^l\xi_l, \quad \xi_l \in \mathfrak{g}_l$$
and the Lie subalgebras of $\Lambda \mathfrak{g}^{\mathbb C}_{\sigma}$ corresponding to the subgroups
$\Lambda G_{\sigma}$, $\Lambda^{+} G^{\mathbb C}_{\sigma}$ and $\Lambda^{-} G^{\mathbb C}_{\sigma}$ are
\begin{eqnarray*}
\Lambda \mathfrak{g}_{\sigma}&=&\Lambda \mathfrak{g}^{\mathbb C}_{\sigma}\cap \mathfrak{g},
\\
\Lambda^{+}\mathfrak{g}^{\mathbb C}_{\sigma}&=&\{\xi\in \Lambda \mathfrak{g}^{\mathbb C}_{\sigma} |
\xi_l=0 \text{ for } l<0, \xi_0\in \mathfrak{k}\},\\
\Lambda^{-}\mathfrak{g}^{\mathbb C}_{\sigma}&=&\{\xi\in \Lambda \mathfrak{g}^{\mathbb C}_{\sigma} |
\xi_l=0 \text{ for } l>0, \xi_0\in \mathfrak{k}\}.
\end{eqnarray*}

Similar conditions hold for the remaining two Lie algebras.

We finish this subsection by quoting the two splitting theorems which are of 
crucial importance for the application of the loop group method.

The first of these theorems is due to Birkhoff, who invented it for the loop group of $GL(n,\C)$ in an attempt to solve Hilbert's 21'{st} problem. 

\begin{theorem}[Birkhoff Decomposition]
Let $G$ be a compact real Lie group. Then the multiplication $\Lambda^{-}_{*} G^{\mathbb C}_{\sigma} \times \Lambda^{+} G^{\mathbb C}_{\sigma} \rightarrow \Lambda G^{\mathbb C}_{\sigma}$ is a 
complex analytic diffeomorphism onto the open, connected  and dense subset 
$\Lambda^{-}_{*} G^{\mathbb C}_{\sigma}\cdot \Lambda^{+} G^{\mathbb C}_{\sigma}$ of  $\Lambda G^{\mathbb C}_{\sigma},$
called the big (left Birkhoff) cell. 

In particular, if $g\in \Lambda G^{\mathbb C}_{\sigma}$ is contained in the big cell,  then
$g$ has a unique decomposition $g=g_{-}g_{+}$, where $g_{-}\in \Lambda_{*}^{-} G_{\sigma}^{\mathbb C}$ 
and $g_{+}\in \Lambda^{+}G^{\mathbb C}_{\sigma}$.
\end{theorem}

The second crucial loop group splitting theorem is the following

\begin{theorem}[Iwasawa decomposition]

Let $G$ be a real compact Lie group. Then the multiplication map 
$$\Lambda G_{\sigma}\times \Lambda^{+}_B G_{\sigma}^{\mathbb C} 
\rightarrow \Lambda G_{\sigma}^{\mathbb C}$$
is a real-analytic diffeomeorphism of Banach Lie groups.
\end{theorem}

This result is well known for untwisted loop groups (see, e.g. \cite{PS})  and was extended to the twisted setting in  \cite{DPW}.  

\begin{acknow}
This work was done mostly during the first named author’s visits at Tsinghua University and the second named author’s visit at the Technical University of Munich. The authors are grateful to both institutions for their generous support and hospitality. The second named author was supported by National Natural Science Foundation of China (Grant Nos. 11831005, 11961131001 and 11671223). 
\end{acknow}

\noindent {\textbf{Remark}} The work leading to this paper and the contents of \cite{DoMaB1} were  originally planned to be contained in a paper with title ``Some new examples of minimal Lagrangian surfaces in $\mathbb{C}P^2$".
 We apologize to all readers for any inconvenience this may have caused them.

%%%%%%%%%%%%%%%%%%%%%%%%%%%%%%%%%%%%%%%%

\end{document}